\documentclass[11pt,twoside]{article}
\usepackage{mathrsfs}
\usepackage{amssymb}
\usepackage{amsmath}
\usepackage[all]{xy}
\usepackage{amsthm}

\usepackage{url}
\numberwithin{equation}{section}

\newtheorem{theorem*}{Theorem}

\theoremstyle{definition}

\theoremstyle{remark}

\usepackage{paralist}

\numberwithin{equation}{section} \theoremstyle{plain}
\newtheorem*{thm*}{Main Theorem}
\newtheorem{thm}{Theorem}[section]
\newtheorem{cor}[thm]{Corollary}
\newtheorem*{cor*}{Corollary}
\newtheorem{lem}[thm]{Lemma}
\newtheorem*{lem*}{Lemma}
\newtheorem{prop}[thm]{Proposition}
\newtheorem*{prop*}{Proposition}
\newtheorem{rem}[thm]{Remark}

\newtheorem*{rem*}{Remark}

\newtheorem*{exa*}{Example}
\newtheorem{df}[thm]{Definition}
\newtheorem*{df*}{Definition}

\newtheorem*{conj*}{Conjecture}

\newtheorem*{Fa*}{Fact}

\newtheorem*{Qu*}{Question}
\newtheorem*{ack*}{ACKNOWLEDGEMENTS}

\newcommand{\pf}{\noindent\begin {proof}}
\newcommand{\epf}{\end{proof}}

\newcommand{\Ext}{\mbox{\rm Ext}}
\newcommand{\Hom}{\mbox{\rm Hom}}
\newcommand{\Tor}{\mbox{\rm Tor}}

\def\Im{\mathop{\rm Im}\nolimits}

\def\mod{\mathop{\rm mod}\nolimits}
\def\Mod{\mathop{\rm Mod}\nolimits}

\def\fd{\mathop{\rm fd}\nolimits}
\def\id{\mathop{\rm id}\nolimits}
\def\pd{\mathop{\rm pd}\nolimits}
\def\max{\mathop{\rm max}\nolimits}

\def\sup{\mathop{\rm sup}\nolimits}
\def\inf{\mathop{\rm inf}\nolimits}

\def\Gfd{\mathop{\rm G\text{-}fd}\nolimits}
\def\Gpd{\mathop{\rm G\text{-}pd}\nolimits}
\def\Gid{\mathop{\rm G\text{-}id}\nolimits}
\def\GCfd{\mathop{\rm G_{\it C}\text{-}fd}\nolimits}
\def\GCpd{\mathop{\rm G_{\it C}\text{-}pd}\nolimits}
\def\GCid{\mathop{\rm G_{\it C}\text{-}id}\nolimits}

\def\Hom{\mathop{\rm Hom}\nolimits}
\def\Ext{\mathop{\rm Ext}\nolimits}
\def\sup{\mathop{\rm sup}\nolimits}

\def\lim{\mathop{\underrightarrow{\rm lim}}\nolimits}

\def\res{\mathop{\rm res}\nolimits}

\def\cores{\mathop{\rm cores}\nolimits}

\def\FPD{\mathop{{\rm FPD}}\nolimits}
\def\FID{\mathop{{\rm FID}}\nolimits}

\begin{document}
\begin{center}
{\large \bf Homological Dimensions Relative to Preresolving Subcategories II}
\footnote{The research was partially supported by NSFC (Grant Nos. 11971225, 12171207).}

\vspace{0.5cm}
Zhaoyong Huang\\
{\footnotesize\it Department of Mathematics, Nanjing University,
Nanjing 210093, Jiangsu Province, P.R. China}\\
{\footnotesize\it E-mail: huangzy@nju.edu.cn}
\end{center}

\bigskip
\centerline { \bf  Abstract}
\bigskip
\leftskip10truemm \rightskip10truemm \noindent
Let $\mathscr{A}$ be an abelian category having enough projective and injective objects,
and let $\mathscr{T}$ be an additive subcategory of $\mathscr{A}$ closed under direct summands.
A known assertion is that in a short exact sequence in $\mathscr{A}$, the $\mathscr{T}$-projective
(respectively, $\mathscr{T}$-injective) dimensions of any two terms can sometimes induce an upper
bound of that of the third term by using the same comparison expressions. We show that
if $\mathscr{T}$ contains all projective (respectively, injective) objects of $\mathscr{A}$, then
the above assertion holds true if and only if $\mathscr{T}$ is resolving (respectively, coresolving).
As applications, we get that a left and right Noetherian ring $R$ is $n$-Gorenstein if and only if
the Gorenstein projective (respectively, injective, flat) dimension of any left $R$-module is
at most $n$. In addition, in several cases, for a subcategory $\mathscr{C}$ of $\mathscr{T}$,
we show that the finitistic $\mathscr{C}$-projective and $\mathscr{T}$-projective dimensions of
$\mathscr{A}$ are identical.
\vbox to 0.3cm{}
\vbox to 0.2cm{}
\leftskip0truemm \rightskip0truemm

\section { \bf Introduction}

Homological dimensions are fundamental invariants in homological theory,
which play a crucial role in studying the structures of modules and rings.
Let $R$ be an arbitrary ring and $\Mod R$ the category of left $R$-modules,
and let $\mathscr{T}$ be a subcategory of $\Mod R$.
For a module $A\in\Mod R$, we use $\mathscr{T}{\text -}\pd A$ to denote the $\mathscr{T}$-projective
dimension of $A$. Let
$$0 \to A_1 \to A_2 \to A_3 \to 0$$
be an exact sequence in $\Mod R$. Consider the following assertions.
\begin{enumerate}
\item[$(1)$] $\mathscr{T}{\text -}\pd A_2\leq
\max\{\mathscr{T}{\text -}\pd A_1,\mathscr{T}{\text -}\pd A_3\}$
with equality if
$\mathscr{T}{\text -}\pd A_1+1\neq\mathscr{T}{\text -}\pd A_3$.
\item[$(2)$] $\mathscr{T}{\text -}\pd A_1\leq
\max\{\mathscr{T}{\text -}\pd A_2,\mathscr{T}{\text -}\pd A_3-1\}$
with equality if
$\mathscr{T}{\text -}\pd A_2\neq\mathscr{T}{\text -}\pd A_3$.
\item[$(3)$] $\mathscr{T}{\text -}\pd A_3\leq
\max\{\mathscr{T}{\text -}\pd A_1+1,\mathscr{T}{\text -}\pd A_2\}$
with equality if
$\mathscr{T}{\text -}\pd A_1\neq\mathscr{T}{\text -}\pd A_2$.
\end{enumerate}
It has been known that these assertions hold true
if $\mathscr{T}$ is the subcategory of $\Mod R$ consisting of one kind of
the following modules: $(a)$ projective modules; $(b)$ flat modules; $(c)$ Gorenstein projective modules
(\cite[Lemma 2.4]{BM}); $(d)$ $C$-Gorenstein projective modules with $C$ a semidualizing bimodule
(\cite[Lemma 3.2]{LHX}); $(e)$ Gorenstein flat modules 
(\cite[Theorem 2.11]{B1} and \cite[Theorem 4.11]{SS}), $(f)$ Auslander classes
(\cite[Corollary 4.5]{Hu3}), and so on. It is natural to ask the following question:
what properties should a subcategory of $\Mod R$ have, in order for properties (1), (2) and (3) to hold?
One of the aims in this paper is to study this question. In fact, we will show that if $\mathscr{T}$
is an additive subcategory of $\Mod R$ which is closed under direct summands and contains
all projective left $R$-modules, then
the above assertions hold true if and only if $\mathscr{T}$ is resolving.

On the other hand, Auslander and Bridger proved that a commutative Noetherian local ring $R$ is Gorenstein
if and only if any finitely generated $R$-module has finite Gorenstein dimension (or Gorenstein projective
dimension in more popular terminology) (\cite[Theorem 4.20]{AB}). Then Hoshino developed
Auslander and Bridger's arguments to prove that an artin algebra $R$ is Gorenstein if and only if
any finitely generated left $R$-module has finite Gorenstein dimension (\cite[Theorem]{Ho}).
Furthermore, Huang and Huang generalized it to left and right Noetherian rings (\cite[Theorem 1.4]{HuH}).
By applying the results obtained by studying the question mentioned above,
our another aim is to generalize this result to arbitrary modules over left and right Noetherian rings.
Note that for a left and right Noetherian ring $R$, if $R$ is $n$-Gorenstein
(that is, the left and right self-injective dimensions of $R$ are at most $n$),
then the Gorenstein projective dimension of any left $R$-module is at most $n$ (\cite[Theorem 11.5.1]{EJ}).
However, the converse seems to be far from clear.

The paper is organized as follows.
In Section 2, we give some notions and notations which will be used in the sequel.

Let $\mathscr{A}$ be an abelian category having enough projective objects.
In Section 3, we first prove the following result.

\begin{thm} \label{thm-1.1} {\rm (Theorem \ref{thm-3.2})}
Let $\mathscr{T}$ be an additive subcategory of $\mathscr{A}$
which is closed under direct summands and contains all
projective objects of $\mathscr{A}$. Then the following statements are equivalent.
\begin{enumerate}
\item[$(1)$] $\mathscr{T}$ is resolving.
\item[$(2)$] For any exact sequence
$$0 \to A_1 \to A_2 \to A_3 \to 0$$
in $\mathscr{A}$, we have
\begin{enumerate}
\item[$(i)$] $\mathscr{T}{\text -}\pd A_2\leq
\max\{\mathscr{T}{\text -}\pd A_1,\mathscr{T}{\text -}\pd A_3\}$ with equality if
$\mathscr{T}{\text -}\pd A_1+1\neq\mathscr{T}{\text -}\pd A_3$.
\item[$(ii)$] $\mathscr{T}{\text -}\pd A_1\leq
\max\{\mathscr{T}{\text -}\pd A_2,\mathscr{T}{\text -}\pd A_3-1\}$ with equality if
$\mathscr{T}{\text -}\pd A_2\neq\mathscr{T}{\text -}\pd A_3$.
\item[$(iii)$] $\mathscr{T}{\text -}\pd A_3\leq
\max\{\mathscr{T}{\text -}\pd A_1+1,\mathscr{T}{\text -}\pd A_2\}$ with equality if
$\mathscr{T}{\text -}\pd A_1\neq\mathscr{T}{\text -}\pd A_2$.
\end{enumerate}
\end{enumerate}
\end{thm}
Then we apply it to prove that
if $\mathscr{T}$ is a resolving subcategory of $\mathscr{A}$ which is closed
under direct summands and admits an $\mathscr{E}$-coproper
cogenerator $\mathscr{C}$ with $\mathscr{E}$ a subcategory of $\mathscr{A}$,
then the finitistic $\mathscr{T}$-projective dimension of $\mathscr{A}$
is at most its finitistic $\mathscr{C}$-projective dimension, and with equality when
$\Ext^{\geq 1}_{\mathscr{A}}(T,C)=0$ for any $T\in\mathscr{T}$ and $C\in\mathscr{C}$
(Corollary \ref{cor-3.5}). We also list the duals of these results without proofs
(Theorem \ref{thm-3.2'} and Corollary \ref{cor-3.5'}).

In Section 4, we first present a partial list of examples of how the results obtained in Section 3 can be applied
(Remark \ref{rem-4.4}). Then it is shown that Corollaries \ref{cor-3.5} and \ref{cor-3.5'} can be applied
in many cases for module categories (Corollaries \ref{cor-4.5}--\ref{cor-4.7}).
Some known results are obtained as corollaries. The main result in this section is the following theorem.

\begin{thm} \label{thm-1.2} {\rm (Theorems \ref{thm-4.9}, \ref{thm-4.11} and \ref{thm-4.13})}
Let $R$ be a left and right Noetherian ring and $n\geq 0$. Then the following statements are equivalent.
\begin{enumerate}
\item[$(1)$] $R$ is $n$-Gorenstein.
\item[$(2)$] The Gorenstein projective dimension of any left $R$-module is at most $n$.
\item[$(3)$] The Gorenstein injective dimension of any left $R$-module is at most $n$.
\item[$(4)$] The Gorenstein flat dimension of any left $R$-module is at most $n$.
\item[$(5)$] The strongly Gorenstein flat dimension of any left $R$-module is at most $n$.
\item[$(6)$] The projectively coresolved Gorenstein flat dimension of any left $R$-module is at most $n$.
\item[$(i)^{op}$] Opposite side version of $(i)$ $(2\leq i\leq 6)$.
\end{enumerate}
\end{thm}

The Gorenstein symmetric conjecture states that for any artin algebra $R$, the left self-injective
dimension of $R$ is finite implies that so is its right self-injective dimension (see \cite[p.410]{ARS}).
By Theorem \ref{thm-1.2}, we have that the Gorenstein symmetric conjecture holds true is equivalent to that
for any artin algebra $R$, the left self-injective dimension of $R$ is at most $n$ implies that any of
(2)--(6) (resp. $(2)^{op}$--$(6)^{op}$) is satisfied.

Let $R,S$ be arbitrary rings and $_RC_S$ a semidualizing bimodule, and let $M\in\Mod R$.
We show that $M$ is $C$-flat if and only if its character module is $C$-injective,
and that $M$ is $C$-Gorenstein flat implies that its character module is $C$-Gorenstein injective
(Theorem \ref{thm-4.17}), which are the $C$-versions of \cite[Theorem 2.2]{B2} and \cite[Theorem 3.6]{H}
respectively. As a consequence, we get that the $C$-Gorenstein flat dimension of $M$ is at most its $C$-flat
dimension with equality if the $C$-flat dimension of $M$ is finite; moreover, the finitistic flat
and Gorenstein flat dimensions of $R$ are identical (Theorem \ref{thm-4.19}). It extends
\cite[Theorem 2.1]{F} and \cite[Theorem 3.24]{H}.

\section {\bf Preliminaries}

Throughout this paper, $\mathscr{A}$ is an abelian category and
all subcategories of $\mathscr{A}$ involved are full, additive
and closed under isomorphisms and direct summands.
We use $\mathcal{P}(\mathscr{A})$ (resp. $\mathcal{I}(\mathscr{A})$)
to denote the subcategory of $\mathscr{A}$ consisting of projective (resp. injective) objects.

Let $\mathscr{X}$ be a subcategory of $\mathscr{A}$.
We write
$${^\perp{\mathscr{X}}}:=\{A\in\mathscr{A}\mid\operatorname{Ext}^{\geq 1}_{\mathscr{A}}(A,X)=0 \mbox{ for any}\ X\in \mathscr{X}\},$$
$${{\mathscr{X}}^\perp}:=\{A\in\mathscr{A}\mid\operatorname{Ext}^{\geq 1}_{\mathscr{A}}(X,A)=0 \mbox{ for any}\ X\in \mathscr{X}\}.$$
Let $M\in\mathscr{A}$.
The \textit{$\mathscr{X}$-projective dimension} $\mathscr{X}$-$\pd M$ of $M$ is defined as
$\inf\{n\mid$ there exists an exact sequence
$$0 \to X_n \to \cdots \to X_1\to X_0 \to M\to 0$$
in $\mathscr{A}$ with all $X_i\in\mathscr{X}\}$, and set $\mathscr{X}$-$\pd M=\infty$
if no such integer exists.
Dually, the \textit{$\mathscr{X}$-injective dimension} $\mathscr{X}$-$\id M$ of $M$ is defined as
$\inf\{n\mid$ there exists an exact sequence
$$0 \to M\to X^0\to X^1\to\cdots \to X^n\to 0$$
in $\mathscr{A}$ with all $X^i\in\mathscr{X}\}$, and set $\mathscr{X}$-$\id M=\infty$
if no such integer exists.
We use $\mathscr{X}$-$\pd^{<\infty}$ (resp. $\mathscr{X}$-$\id^{<\infty}$) to denote the subcategory of
$\mathscr{A}$ consisting of objects with finite $\mathscr{X}$-projective (resp. $\mathscr{X}$-injective) dimension.
We write
$$\mathscr{X}\text{-}\FPD:=\sup\{\mathscr{X}\text{-}\pd M\mid M\in\mathscr{X}\text{-}\pd^{<\infty}\},$$
$$\mathscr{X}\text{-}\FID:=\sup\{\mathscr{X}\text{-}\id M\mid M\in\mathscr{X}\text{-}\id^{<\infty}\}.$$

Let $\mathscr{E}$ be a subcategory of $\mathscr{A}$. Recall from \cite{EJ} that a sequence
$$\mathbb{S}: \cdots \to S_1 \to S_2 \to S_3 \to \cdots$$
in $\mathscr{A}$ is called {\it $\Hom_{\mathscr{A}}(\mathscr{E},-)$-exact}
(resp. {\it $\Hom_{\mathscr{A}}(-,\mathscr{E})$-exact})
if {\it $\Hom_{\mathscr{A}}(E,\mathbb{S})$} (resp. $\Hom_{\mathscr{A}}(\mathbb{S},E)$)
is exact for any $E\in\mathscr{E}$.
Let $\mathscr{C}\subseteq\mathscr{T}$ be subcategories of $\mathscr{A}$.
Recall from \cite{Hu2} that $\mathscr{C}$ is called an {\it $\mathscr{E}$-proper generator}
(resp. {\it $\mathscr{E}$-coproper cogenerator}) for $\mathscr{T}$ if for any $T\in\mathscr{T}$, there exists a
$\Hom_{\mathscr{A}}(\mathscr{E},-)$ (resp. $\Hom_{\mathscr{A}}(-,\mathscr{E})$)-exact exact sequence
$$0\to T^{'}\to C \to T \to 0\ {\rm (resp.}\ 0\to T\to C \to T^{'} \to 0)$$ in
$\mathscr{A}$ with $C\in\mathscr{C}$ and $T^{'}\in\mathscr{T}$. When $\mathscr{E}=\mathcal{P}(\mathscr{A})$
(resp. $\mathcal{I}(\mathscr{A})$), an $\mathscr{E}$-proper generator (resp. $\mathscr{E}$-coproper cogenerator)
is exactly a usual generator (resp. cogenerator).

We define $\widetilde{\res_{\mathscr{E}}{\mathscr{C}}}:=\{M\in\mathscr{A}\mid$ there exists a
$\Hom_{\mathscr{A}}(\mathscr{E},-)$-exact exact sequence
$$\cdots \to C_i \to \cdots \to C_1 \to C_0 \to M\to 0$$
in $\mathscr{A}$ with all $C_i\in\mathscr{C}\}$. Dually, we define
$\widetilde{\cores_{\mathscr{E}}{\mathscr{C}}}:=\{M\in\mathscr{A}\mid$ there exists a
$\Hom_{\mathscr{A}}(-,\mathscr{E})$-exact exact sequence
$$0 \to M\to C^0 \to C^1\to \cdots \to C^i \to \cdots$$ in $\mathscr{A}$ with all
$C^i$ in $\mathscr{C}\}$. 

\begin{df}\label{def-2.1} {\rm (\cite{Hu2})
Let $\mathscr{E}$ and $\mathscr{T}$ be subcategories of $\mathscr{A}$.
\begin{enumerate}
\item[(1)] The subcategory $\mathscr{T}$ is called {\it $\mathscr{E}$-preresolving}
in $\mathscr{A}$ if the following conditions are satisfied.
\begin{enumerate}
\item[(1.1)] $\mathscr{T}$ admits an $\mathscr{E}$-proper generator.
\item[(1.2)] $\mathscr{T}$ is {\it closed under $\mathscr{E}$-proper
extensions}, that is, for any
$\Hom_{\mathscr{A}}(\mathscr{E},-)$-exact exact sequence
$$0\to A_1\to A_2 \to A_3 \to 0$$ in $\mathscr{A}$, if both $A_1$ and $A_3$
are in $\mathscr{T}$, then $A_2$ is also in $\mathscr{T}$.
\end{enumerate}
\item[(2)] The subcategory $\mathscr{T}$ is called {\it
$\mathscr{E}$-precoresolving} in $\mathscr{A}$ if the following
conditions are satisfied.
\begin{enumerate}
\item[(2.1)] $\mathscr{T}$ admits an $\mathscr{E}$-coproper cogenerator.
\item[(2.2)] $\mathscr{T}$ is {\it closed under $\mathscr{E}$-coproper
extensions}, that is, for any
$\Hom_{\mathscr{A}}(-,\mathscr{E})$-exact exact sequence
$$0\to A_1\to A_2 \to A_3 \to 0$$ in $\mathscr{A}$, if both $A_1$ and $A_3$
are in $\mathscr{T}$, then $A_2$ is also in $\mathscr{T}$.
\end{enumerate}
\end{enumerate}}
\end{df}

The following definition is cited from \cite{EO}.

\begin{df}\label{def-2.2}
{\rm Let $\mathscr{U},\mathscr{V}$ be subcategories of $\mathscr{A}$.
\begin{enumerate}
\item[(1)] The pair $(\mathscr{U},\mathscr{V})$ is called a {\it cotorsion pair} in $\mathscr{A}$ if
$\mathscr{U}=\{A\in\mathscr{A}\mid\Ext_{\mathscr{A}}^1(A,V)=0$ for any $V\in\mathscr{V}\}$ and
$\mathscr{V}=\{A\in\mathscr{A}\mid\Ext_{\mathscr{A}}^1(U,A)=0$ for any $U\in\mathscr{U}\}$.
\item[(2)] A cotorsion pair $(\mathscr{U},\mathscr{V})$ is called {\it hereditary} if one of the following equivalent
conditions is satisfied.
\begin{enumerate}
\item[(2.1)] $\Ext_{\mathscr{A}}^{\geq 1}(U,V)=0$ for any $U\in\mathscr{U}$ and $V\in\mathscr{V}$.
\item[(2.2)] $\mathscr{U}$ is resolving in the sense that $\mathcal{P}(\mathscr{A})\subseteq\mathscr{U}$
and $\mathscr{U}$ is closed under extensions and kernels of epimorphisms.
\item[(2.3)] $\mathscr{V}$ is coresolving in the sense that $\mathcal{I}(\mathscr{A})\subseteq\mathscr{V}$
and $\mathscr{V}$ is closed under extensions and cokernels of monomorphisms.
\end{enumerate}
\end{enumerate}}
\end{df}

\section {\bf General results}



\subsection {Projective dimension relative to resolving subcategories}

We begin with the following observation.

\begin{lem} \label{lem-3.1}
Let $M\in\mathscr{A}$ and $n\geq 0$.
\begin{enumerate}
\item[$(1)$] Assume that $\mathscr{A}$ has enough projective objects.
If $\mathscr{T}$ is a resolving subcategory of $\mathscr{A}$, then
the following statements are equivalent.
\begin{enumerate}
\item[$(1.1)$] $\mathscr{T}{\text -}\pd M\leq n$.
\item[$(1.2)$] There exists an exact sequence
$$0\to K_n\to P_{n-1}\to\cdots\to P_1\to P_0\to M\to 0$$
in $\mathscr{A}$ with all $P_i$ projective and $K_n\in\mathscr{T}$.
\item[$(1.3)$] For any exact sequence
$$0\to K_n\to P_{n-1}\to\cdots\to P_1\to P_0\to M\to 0$$
in $\mathscr{A}$, if all $P_i$ are projective, then
$K_n\in\mathscr{T}$.
\item[$(1.4)$] For any exact sequence
$$0\to K_n\to T_{n-1}\to\cdots\to T_1\to T_0\to M\to 0$$
in $\mathscr{A}$, if all $T_i$ are in $\mathscr{T}$, then
$K_n\in\mathscr{T}$.
\end{enumerate}
\item[$(2)$] Let $\mathscr{E}$ be a subcategory of $\mathscr{A}$.
If $\mathscr{T}$ is an $\mathscr{E}$-precoresolving subcategory of $\mathscr{A}$
admitting an $\mathscr{E}$-coproper cogenerator $\mathscr{C}$, then
the following statements are equivalent.
\begin{enumerate}
\item[$(2.1)$] $\mathscr{T}{\text -}\pd M\leq n$.
\item[$(2.2)$] There exists an exact sequence
$$0\to C_n\to C_{n-1}\to\cdots\to C_1\to T_0\to M\to 0$$
in $\mathscr{A}$ with all $C_i$ in $\mathscr{C}$ and $T_0\in\mathscr{T}$;
that is, there exists an exact sequence
$$0\to K\to T\to M\to 0$$
in $\mathscr{A}$ with $T\in\mathscr{T}$ and $\mathscr{C}$-$\pd K\leq n-1$.
\end{enumerate}
\end{enumerate}
\end{lem}

\begin{proof}
(1) The implications $(1.4)\Rightarrow (1.3)\Rightarrow (1.2)$ are trivial.
By \cite[Theorem 3.6]{Hu2}, we have $(1.1)\Leftrightarrow (1.2)$.
By \cite[Lemma 3.12]{AB}, we have $(1.1)\Rightarrow (1.4)$.

(2) It follows from \cite[Theorem 4.7]{Hu2}.
\end{proof}


The main result in this subsection is as follows.

\begin{thm} \label{thm-3.2}
Assume that $\mathscr{A}$ has enough projective objects and
$\mathscr{T}$ is a subcategory of $\mathscr{A}$ containing $\mathcal{P}(\mathscr{A})$.
Then the following statements are equivalent.
\begin{enumerate}
\item[$(1)$] $\mathscr{T}$ is resolving.
\item[$(2)$] For any exact sequence
$$0 \to A_1 \to A_2 \to A_3 \to 0$$
in $\mathscr{A}$, we have
\begin{enumerate}
\item[$(i)$] $(a)$ $\mathscr{T}{\text -}\pd A_2\leq
\max\{\mathscr{T}{\text -}\pd A_1,\mathscr{T}{\text -}\pd A_3\}$,
$(b)$ the equality holds if
$\mathscr{T}{\text -}\pd A_1+1\neq\mathscr{T}{\text -}\pd A_3$.
\item[$(ii)$] $(a)$ $\mathscr{T}{\text -}\pd A_1\leq
\max\{\mathscr{T}{\text -}\pd A_2,\mathscr{T}{\text -}\pd A_3-1\}$,
$(b)$ the equality holds if
$\mathscr{T}{\text -}\pd A_2\neq\mathscr{T}{\text -}\pd A_3$.
\item[$(iii)$] $(a)$ $\mathscr{T}{\text -}\pd A_3\leq
\max\{\mathscr{T}{\text -}\pd A_1+1,\mathscr{T}{\text -}\pd A_2\}$,
$(b)$ the equality holds if
$\mathscr{T}{\text -}\pd A_1\neq\mathscr{T}{\text -}\pd A_2$.
\end{enumerate}
\end{enumerate}
\end{thm}

\begin{proof}
$(2)\Rightarrow (1)$ By $(i)(a)$ and $(ii)(a)$, we have that $\mathscr{T}$
is closed extensions and kernels of epimorphisms respectively, and so $\mathscr{T}$ is resolving.

$(1)\Rightarrow (2)$ $(i)(a)$ If $\max\{\mathscr{T}{\text -}\pd A_1,\mathscr{T}{\text -}\pd A_3\}=0$,
that is, both $A_1$ and $A_3$ are in $\mathscr{T}$, then $A_2$ is also in $\mathscr{T}$ by (1),
and the assertion follows. Now suppose $\max\{\mathscr{T}{\text -}\pd A_1,\mathscr{T}{\text -}\pd A_3\}=n\geq 1$.
By Lemma \ref{lem-3.1}(1), we have the following two exact sequences
$$0\to K^{'}_n\to P^{'}_{n-1}\to\cdots P^{'}_{1}\to P^{'}_{0}\to A_1\to 0,$$
$$0\to K^{''}_n\to P^{''}_{n-1}\to\cdots P^{''}_{1}\to P^{''}_{0}\to A_3\to 0$$
in $\mathscr{A}$ with all $P^{'}_i,P^{''}_i$ projective and $K^{'}_n,K^{''}_n\in\mathscr{T}$.
Then by the horseshoe lemma, we get the following two exact sequences
$$0\to K_n\to P^{'}_{n-1}\oplus P^{''}_{n-1}\to\cdots P^{'}_{1}\oplus P^{''}_{1}
\to P^{'}_{0}\oplus P^{''}_{0}\to A_2\to 0,\eqno{(3.1)}$$
$$0\to K^{'}_n\to K_n\to K^{''}_n\to 0. \eqno{(3.2)}$$
By the exact sequence (3.2) and (1), we have $K_n\in\mathscr{T}$.
Then the exact sequence (3.1) implies $\mathscr{T}{\text -}\pd A_2\leq n$.

$(ii)(a)$
Let $\mathscr{T}{\text -}\pd A_2=n_2$ and $\mathscr{T}{\text -}\pd A_3=n_3$ with $n_2,n_3<\infty$.

We first suppose $n_3=0$ (that is, $A_3\in\mathscr{T}$). If $n_2=0$ (that is, $A_2\in\mathscr{T}$),
then $A_1\in\mathscr{T}$ by (1). If $n_2\geq 1$, then by Lemma \ref{lem-3.1}(1), there exists an exact sequence
$$0\to A^{'}_2\to P\to A_2\to 0$$
in $\mathscr{A}$ with $P$ projective and $\mathscr{T}{\text -}\pd A^{'}_2\leq n_2-1$.
Consider the following pull-back diagram
$$\xymatrix{ & 0 \ar[d]& 0 \ar[d] & \\
& A^{'}_2 \ar@{-->}[d] \ar@{==}[r]& A^{'}_2 \ar[d] & \\
0 \ar@{-->}[r] & T \ar@{-->}[r] \ar@{-->}[d]& P \ar@{-->}[r] \ar[d] & A_3 \ar@{-->}[r] \ar@{==}[d]& 0\\
0 \ar[r] &  A_1 \ar[r] \ar@{-->}[d]& A_2 \ar[r] \ar[d] & A_3 \ar[r] &0\\
& 0 & 0. &}$$
By (1) and middle row in the above diagram, we have $T\in\mathscr{T}$.
Then the leftmost column in this diagram implies $\mathscr{T}{\text -}\pd A_1\leq n_2$.

Now suppose $n_3\geq 1$. Then by Lemma \ref{lem-3.1}(1), there exists an exact sequence
$$0\to A^{'}_3\to Q\to A_3\to 0$$
in $\mathscr{A}$ with $Q$ projective and $\mathscr{T}{\text -}\pd A^{'}_3\leq n_3-1$.
Consider the following pull-back diagram
$$\xymatrix{& & 0 \ar@{-->}[d] & 0 \ar[d]& &\\
& & A^{'}_3 \ar@{==}[r] \ar@{-->}[d] & A^{'}_3 \ar[d]& &\\
0 \ar@{-->}[r] & A_1 \ar@{==}[d] \ar@{-->}[r] & A_1\oplus Q \ar@{-->}[d]
\ar@{-->}[r] & Q\ar[d] \ar@{-->}[r] & 0\\
0 \ar[r] & A_1 \ar[r] & A_2 \ar@{-->}[d] \ar[r] & A_3 \ar[d] \ar[r] & 0\\
& & 0 & 0. & & }$$
By $(i)(a)$ and middle column in the above diagram, we have
$\mathscr{T}{\text -}\pd (A_1\oplus Q)\leq\max\{n_2,n_3-1\}$. It follows
from \cite[Corollary 3.9]{Hu2} that $\mathscr{T}{\text -}\pd A_1\leq\max\{n_2,n_3-1\}$.


$(iii)(a)$ Let $\mathscr{T}{\text -}\pd A_1=n_1$ and $\mathscr{T}{\text -}\pd A_2=n_2$ with $n_1,n_2<\infty$.
If $n_2=0$, that is, $A_2\in\mathscr{T}$, then $\mathscr{T}{\text -}\pd A_3=n_1+1$. Now suppose
$n_2\geq 1$. Then by Lemma \ref{lem-3.1}(1), there exists an exact sequence
$$0\to A^{'}_2\to P\to A_2\to 0$$
in $\mathscr{A}$ with $P$ projective and $\mathscr{T}{\text -}\pd A^{'}_2\leq n_2-1$.
Consider the following pull-back diagram
$$\xymatrix{ & 0 \ar[d]& 0 \ar[d] & \\
& A^{'}_2 \ar@{-->}[d] \ar@{==}[r]& A^{'}_2 \ar[d] & \\
0 \ar@{-->}[r] & K \ar@{-->}[r] \ar@{-->}[d]& P \ar@{-->}[r] \ar[d] & A_3 \ar@{-->}[r] \ar@{==}[d]& 0\\
0 \ar[r] &  A_1 \ar[r] \ar@{-->}[d]& A_2 \ar[r] \ar[d] & A_3 \ar[r] &0\\
& 0 & 0. &}$$
By (1) and the leftmost column in the above diagram, we have
$$\mathscr{T}{\text -}\pd K\leq
\max\{\mathscr{T}{\text -}\pd A_1,\mathscr{T}{\text -}\pd A^{'}_2\}\leq\max\{n_1,n_2-1\}.$$
Then the middle row in this diagram implies
$$\mathscr{T}{\text -}\pd A_3\leq
\mathscr{T}{\text -}\pd K+1\leq\max\{n_1,n_2-1\}+1=\max\{n_1+1,n_2\}.$$


$(i)(b)$ If $\mathscr{T}{\text -}\pd A_1+1<\mathscr{T}{\text -}\pd A_3$, then
$\mathscr{T}{\text -}\pd A_2\leq\mathscr{T}{\text -}\pd A_3$ and
$\mathscr{T}{\text -}\pd A_3\leq\mathscr{T}{\text -}\pd A_2$
by $(i)(a)$ and $(iii)(a)$ respectively.
Thus $\mathscr{T}{\text -}\pd A_2=\mathscr{T}{\text -}\pd A_3$.

If $\mathscr{T}{\text -}\pd A_3<\mathscr{T}{\text -}\pd A_1+1$, then
$\mathscr{T}{\text -}\pd A_2\leq\mathscr{T}{\text -}\pd A_1$ and
$\mathscr{T}{\text -}\pd A_1\leq\mathscr{T}{\text -}\pd A_2$ by $(i)(a)$
and $(ii)(a)$ respectively.
Thus $\mathscr{T}{\text -}\pd A_2=\mathscr{T}{\text -}\pd A_1$.

$(ii)(b)$ If $\mathscr{T}{\text -}\pd A_2<\mathscr{T}{\text -}\pd A_3$,
then $\mathscr{T}{\text -}\pd A_1\leq\mathscr{T}{\text -}\pd A_3-1$ and
$\mathscr{T}{\text -}\pd A_3\leq\mathscr{T}{\text -}\pd A_1+1$
by $(ii)(a)$ and $(iii)(a)$ respectively.
Thus $\mathscr{T}{\text -}\pd A_1=\mathscr{T}{\text -}\pd A_3-1$.

If $\mathscr{T}{\text -}\pd A_3<\mathscr{T}{\text -}\pd A_2$, then
$\mathscr{T}{\text -}\pd A_1\leq\mathscr{T}{\text -}\pd A_2$ and
$\mathscr{T}{\text -}\pd A_2\leq\mathscr{T}{\text -}\pd A_1$
by $(ii)(a)$ and $(i)(a)$ respectively.
Thus $\mathscr{T}{\text -}\pd A_1=\mathscr{T}{\text -}\pd A_2$.

$(iii)(b)$ If $\mathscr{T}{\text -}\pd A_1<\mathscr{T}{\text -}\pd A_2$,
then $\mathscr{T}{\text -}\pd A_3\leq\mathscr{T}{\text -}\pd A_2$ and
$\mathscr{T}{\text -}\pd A_2\leq\mathscr{T}{\text -}\pd A_3$
by $(iii)(a)$ and $(i)(a)$ respectively.
Thus $\mathscr{T}{\text -}\pd A_3=\mathscr{T}{\text -}\pd A_2$.

If $\mathscr{T}{\text -}\pd A_2<\mathscr{T}{\text -}\pd A_1$, then
$\mathscr{T}{\text -}\pd A_3\leq\mathscr{T}{\text -}\pd A_1+1$ and
$\mathscr{T}{\text -}\pd A_1\leq\mathscr{T}{\text -}\pd A_3-1$
by $(iii)(a)$ and $(ii)(a)$ respectively.
Thus $\mathscr{T}{\text -}\pd A_3=\mathscr{T}{\text -}\pd A_1+1$.
\end{proof}

As an immediate consequence, we get the following result.

\begin{cor} \label{cor-3.3}
Assume that $\mathscr{A}$ has enough projective objects and
$\mathscr{T}$ is a resolving subcategory of $\mathscr{A}$. Then
$\mathscr{T}$-$\pd^{<\infty}$ satisfies the two-out-of-three property; that is,
in a short exact sequence in $\mathscr{A}$,
if any two terms are in $\mathscr{T}$-$\pd^{<\infty}$, then so is the third term.
\end{cor}

The following result shows that if the resolving subcategory $\mathscr{T}$
of $\mathscr{A}$ admits an $\mathscr{E}$-coproper
cogenerator $\mathscr{C}$, then any object in $\mathscr{A}$ with finite
$\mathscr{T}$-projective dimension is isomorphic to a kernel (respectively, a cokernel)
of a morphism from an object  in $\mathscr{A}$
with finite $\mathscr{C}$-projective dimension to an object in $\mathscr{T}$.

\begin{cor}\label{cor-3.4}
Let $\mathscr{E}$ be a subcategory of $\mathscr{A}$. If $\mathscr{T}$ is an $\mathscr{E}$-precoresolving
subcategory of $\mathscr{A}$ admitting an $\mathscr{E}$-coproper cogenerator $\mathscr{C}$,
then for any $M\in\mathscr{A}$ with $\mathscr{T}{\text -}\pd M=n<\infty$, we have
\begin{enumerate}
\item[$(1)$] There exists an exact sequence
$$0\to K\to T\to K^{'}\to T^{'}\to 0$$
in $\mathscr{A}$ with $\mathscr{C}{\text -}\pd K\leq n-1$, $\mathscr{C}{\text -}\pd K^{'}\leq n$
and $T,T^{'}\in\mathscr{T}$, such that $M\cong\Im(T\to K^{'})$.
\item[$(2)$] If $\mathscr{A}$ has enough projective objects and
$\mathscr{T}$ is resolving in $\mathscr{A}$, then the two ``$\leq$" in $(1)$ are ``$=$".
\end{enumerate}
\end{cor}

\begin{proof}
(1) Let $M\in\mathscr{A}$ with $\mathscr{T}{\text -}\pd M=n<\infty$. The case for $n=0$ is trivial.
Now suppose $n\geq 1$. By Lemma \ref{lem-3.1}(2), there exists an exact sequence
$$0\to K\to T\to M\to 0\eqno{(3.3)}$$
in $\mathscr{A}$ with $\mathscr{C}{\text -}\pd K\leq n-1$ and $T\in\mathscr{T}$.
Thus there exists an exact sequence
$$0\to T\to C\to T^{'}\to 0$$
in $\mathscr{A}$ with $C\in\mathscr{C}$ and $T^{'}\in\mathscr{T}$.
Consider the following push-out diagram
$$\xymatrix{&  &  0 \ar[d] & 0 \ar[d] & \\
0\ar[r] & K\ar[r] \ar@{==}[d]& T \ar[r] \ar[d]& M \ar[r] \ar@{-->}[d]& 0\\
0\ar@{-->}[r] & K\ar@{-->}[r] & C \ar@{-->}[r] \ar[d]& K^{'} \ar@{-->}[r] \ar@{-->}[d]& 0\\
& & T^{'}\ar@{==}[r] \ar[d]&T^{'} \ar@{-->}[d]\\
& & 0 & 0.& }$$
By the middle row in the above diagram, we have $\mathscr{C}{\text -}\pd K^{'}\leq n$.
Now splicing (3.3) and the rightmost column
$$0\to M\to K^{'}\to T^{'}\to 0,\eqno{(3.4)}$$
we get the desired exact sequence.

(2) Assume that $\mathscr{A}$ has enough projective objects and
$\mathscr{T}$ is resolving in $\mathscr{A}$. Then by (3.3) and Theorem \ref{thm-3.2}(2)$(ii)$,
we have $\mathscr{T}{\text -}\pd K=n-1$. Since $\mathscr{C}{\text -}\pd K\geq\mathscr{T}{\text -}\pd K$,
we have $\mathscr{C}{\text -}\pd K=n-1$. By (3.4) and Theorem \ref{thm-3.2}(2)$(i)$,
we have $\mathscr{T}{\text -}\pd K^{'}=n$, and so $\mathscr{C}{\text -}\pd K^{'}=n$.
\end{proof}

Furthermore, we get the following result.

\begin{cor} \label{cor-3.5}
Assume that $\mathscr{A}$ has enough projective objects and
$\mathscr{T}$ is a resolving subcategory of $\mathscr{A}$ admitting an $\mathscr{E}$-coproper
cogenerator $\mathscr{C}$. Then
\begin{enumerate}
\item[$(1)$] $\mathscr{T}{\text -}\FPD\leq\mathscr{C}{\text -}\FPD$.
\item[$(2)$] If $\mathscr{T}\subseteq{^{\bot}\mathscr{C}}$, then
$\mathscr{T}{\text -}\pd M=\mathscr{C}{\text -}\pd M$ for any $M\in\mathscr{A}$ with
$\mathscr{C}{\text -}\pd M<\infty$.
\item[$(3)$] If $\mathscr{T}\subseteq{^{\bot}\mathscr{C}}$,
then $\mathscr{T}{\text -}\FPD=\mathscr{C}{\text -}\FPD$.
\end{enumerate}
\end{cor}

\begin{proof}
(1) Let $M\in\mathscr{A}$ with $\mathscr{T}{\text -}\pd M=n<\infty$. Then by Corollary \ref{cor-3.4},
there exists $K^{'}\in\mathscr{A}$ such that $\mathscr{C}{\text -}\pd K^{'}=n$. It follows that
$\mathscr{T}{\text -}\FPD\leq\mathscr{C}{\text -}\FPD$.

(2) Let $M\in\mathscr{A}$ with $\mathscr{C}{\text -}\pd M=n<\infty$. Then $\mathscr{T}{\text -}\pd M=m\leq n$.
By Corollary \ref{cor-3.4}, there exists an exact sequence
$$0\to M\to K^{'}\to T^{'}\to 0$$
in $\mathscr{A}$ with $\mathscr{C}{\text -}\pd K^{'}=m$ and $T^{'}\in\mathscr{T}$.
Since $\mathscr{T}\subseteq{^{\bot}\mathscr{C}}$ by assumption, we have $\Ext^{\geq 1}_{\mathscr{A}}(T^{'},M)=0$
by dimension shifting. So the above exact sequence splits and $M$ is isomorphic to a direct summand of $K^{'}$.
So $n=\mathscr{C}{\text -}\pd M\leq m$ by \cite[Corollary 3.9]{Hu2}, and hence $m=n$ and $\mathscr{T}{\text -}\pd M=n$.

(3) By (2), we have $\mathscr{C}{\text -}\FPD\leq\mathscr{T}{\text -}\FPD$. So the assertion follows from (1).
\end{proof}


In the next section, we need the following two propositions.

\begin{prop} \label{prop-3.6}
Let $\mathscr{E}$ and $\mathscr{C}$ be subcategories of $\mathscr{A}$.
If $^{\bot}\mathscr{E}\cap\widetilde{\cores_{\mathscr{E}}\mathscr{C}}$
is closed under ($\mathscr{E}$-coproper) extensions, then it is
closed under kernels of epimorphisms. In particular, if
$\cores{\mathscr{C}}:=\widetilde{\cores_{\mathcal{I}(\mathscr{A})}{\mathscr{C}}}$
is closed under extensions, then it is closed under kernels of epimorphisms.
\end{prop}

\begin{proof}
Let
$$0\to A\to T_1\to T_2\to 0$$
be an exact sequence in $\mathscr{A}$ with $T_1,T_2\in
{^{\bot}\mathscr{E}\cap\widetilde{\cores_{\mathscr{E}}\mathscr{C}}}$.
Then there exists a $\Hom_{\mathscr{A}}(-,\mathscr{E})$-exact exact sequence
$$0\to T_1\to C\to T^{'}_1\to 0$$
in $\mathscr{A}$ with $C\in\mathscr{C}$ and
$T^{'}_1\in{^{\bot}\mathscr{E}\cap\widetilde{\cores_{\mathscr{E}}\mathscr{C}}}$.
Consider the following push-out diagram
$$\xymatrix{&  &  0 \ar[d] & 0 \ar[d] & \\
0\ar[r] & A\ar[r] \ar@{==}[d]& T_1 \ar[r] \ar[d]& T_2 \ar[r] \ar@{-->}[d]& 0\\
0\ar@{-->}[r] & A\ar@{-->}[r] & C \ar@{-->}[r] \ar[d]& T \ar@{-->}[r] \ar@{-->}[d]& 0\\
& & T^{'}_1\ar@{==}[r] \ar[d]&T^{'}_1\ar@{-->}[d]\\
& & 0 & 0.& }$$
By \cite[Lemma 2.4(2)]{Hu1}, all columns and rows in this diagram
are $\Hom_{\mathscr{A}}(-,\mathscr{E})$-exact exact sequences. If
$^{\bot}\mathscr{E}\cap\widetilde{\cores_{\mathscr{E}}\mathscr{C}}$ is closed under
$\mathscr{E}$-coproper extensions, then the rightmost column implies
$T\in{^{\bot}\mathscr{E}\cap\widetilde{\cores_{\mathscr{E}}\mathscr{C}}}$,
and thus the middle row yields
$A\in{^{\bot}\mathscr{E}\cap\widetilde{\cores_{\mathscr{E}}\mathscr{C}}}$.

The latter assertion follows from the former one by putting
$\mathscr{E}=\mathcal{I}(\mathscr{A})$.
\end{proof}

\begin{prop} \label{prop-3.7}
Let $\mathscr{E}$ be a subcategory of $\mathscr{A}$. If $\mathscr{T}$ is an
$\mathscr{E}$-precoresolving subcategory of $\mathscr{A}$ admitting an $\mathscr{E}$-coproper
cogenerator $\mathscr{C}$, then
$\widetilde{\cores_{\mathscr{E}}\mathscr{C}}=\widetilde{\cores_{\mathscr{E}}\mathscr{T}}$.
\end{prop}

\begin{proof}
It is trivial that $\widetilde{\cores_{\mathscr{E}}\mathscr{C}}\subseteq\widetilde{\cores_{\mathscr{E}}\mathscr{T}}$.
Now let $M\in\widetilde{\cores_{\mathscr{E}}\mathscr{T}}$ and let
$$0\to M\to T\to M^{'}\to 0$$
be a $\Hom_{\mathscr{A}}(-,\mathscr{E})$-exact exact sequence in $\mathscr{A}$ with $T\in\mathscr{T}$
and $M^{'}\in\widetilde{\cores_{\mathscr{E}}\mathscr{T}}$.
Since $\mathscr{T}$ admits $\mathscr{E}$-coproper cogenerator
$\mathscr{C}$ by assumption, there exists a $\Hom_{\mathscr{A}}(-,\mathscr{E})$-exact exact sequence
$$0\to T\to C^{0} \to T^{'}\to 0$$
in $\mathscr{A}$ with $C^{0}\in\mathscr{C}$ and in $T^{'}\in\mathscr{T}$.
Then we have the following push-out diagram
$$\xymatrix{&  &  0 \ar[d] & 0 \ar[d] & \\
0\ar[r] & M\ar[r] \ar@{==}[d]& T \ar[r] \ar[d]& M^{'} \ar[r] \ar@{-->}[d]& 0\\
0\ar@{-->}[r] & M\ar@{-->}[r] & C^{0} \ar@{-->}[r] \ar[d]& M^{1} \ar@{-->}[r] \ar@{-->}[d]& 0\\
& & T^{'}\ar@{==}[r] \ar[d]&T^{'} \ar@{-->}[d]\\
& & 0 & 0.& }$$
Since there also exists a $\Hom_{\mathscr{A}}(-,\mathscr{E})$-exact exact sequence
$$0\to M^{'}\to T^{''}\to M^{''}\to 0$$
in $\mathscr{A}$ with $T^{''}\in\mathscr{T}$ and $M^{''}\in\widetilde{\cores_{\mathscr{E}}\mathscr{T}}$,
we have the following push-out diagram
$$\xymatrix{&  0 \ar[d] & 0 \ar[d] & & \\
0\ar[r] & M^{'}\ar[r] \ar[d]& T^{''} \ar[r] \ar@{-->}[d]& M^{''} \ar[r] \ar@{==}[d]& 0\\
0\ar@{-->}[r] & M^{1}\ar@{-->}[r]\ar[d] & T^{1} \ar@{-->}[r] \ar@{-->}[d]& M^{''} \ar@{-->}[r] & 0\\
& T^{'}\ar@{==}[r] \ar[d]&T^{'} \ar@{-->}[d]&\\
& 0 & 0.&& }$$
It follows from \cite[Lemma 2.4(2)]{Hu1} that all columns and rows in the above two diagrams are
$\Hom_{\mathscr{A}}(-,\mathscr{E})$-exact exact sequences.
Since $\mathscr{T}$ is closed under $\mathscr{E}$-coproper extensions
by assumption, the middle column in the second diagram implies $T^{1}\in\mathscr{T}$, and hence the middle row
in this diagram implies $M^{1}\in\widetilde{\cores_{\mathscr{E}}\mathscr{T}}$. Similarly, we get a
$\Hom_{\mathscr{A}}(-,\mathscr{E})$-exact exact sequence
$$0\to M^{1}\to C^{1}\to M^{2}\to 0$$
in $\mathscr{A}$ with $C^{1}\in\mathscr{C}$ and $M^{2}\in\widetilde{\cores_{\mathscr{E}}\mathscr{T}}$.
Continuing this process, we get a $\Hom_{\mathscr{A}}(-,\mathscr{E})$-exact exact sequence
$$0\to M\to C^{0}\to C^{1}\to\cdots\to C^i\to\cdots$$
in $\mathscr{A}$ with all $C^{i}$ in $\mathscr{C}$. It follows that $M\in\widetilde{\cores_{\mathscr{E}}\mathscr{C}}$
and $\widetilde{\cores_{\mathscr{E}}\mathscr{T}}\subseteq\widetilde{\cores_{\mathscr{E}}\mathscr{C}}$.
\end{proof}

\subsection {Injective dimension relative to coresolving subcategories}

All results and their proofs in this subsection are completely dual to those in Subsection 3.1,
so we only list the results without proofs.

\begin{lem} \label{lem-3.1'}
Let $M\in\mathscr{A}$ and $n\geq 0$.
\begin{enumerate}
\item[$(1)$] Assume that $\mathscr{A}$ has enough injective objects.
If $\mathscr{T}$ is a coresolving subcategory of $\mathscr{A}$, then
the following statements are equivalent.
\begin{enumerate}
\item[$(1.1)$] $\mathscr{T}{\text -}\id M\leq n$.
\item[$(1.2)$] There exists an exact sequence
$$0\to M\to I^0\to I^1\to\cdots\to I^{n-1}\to K^n\to 0$$
in $\mathscr{A}$ with all $I^i$ injective and $K^n\in\mathscr{T}$.
\item[$(1.3)$] For any exact sequence
$$0\to M\to I^0\to I^1\to\cdots\to I^{n-1}\to K^n\to 0$$
in $\mathscr{A}$, if all $I^i$ are injective, then $K^n\in\mathscr{T}$.
\item[$(1.4)$] For any exact sequence
$$0\to M\to T^0\to T^1\to\cdots\to T^{n-1}\to K^n\to 0$$
in $\mathscr{A}$, if all $T^i$ are in $\mathscr{T}$, then
$K^n\in\mathscr{T}$.
\end{enumerate}
\item[$(2)$] Let $\mathscr{E}$ be a subcategory of $\mathscr{A}$.
If $\mathscr{T}$ is an $\mathscr{E}$-preresolving subcategory of $\mathscr{A}$
admitting an $\mathscr{E}$-proper generator $\mathscr{C}$, then
the following statements are equivalent.
\begin{enumerate}
\item[$(2.1)$] $\mathscr{T}{\text -}\id M\leq n$.
\item[$(2.2)$] There exists an exact sequence
$$0\to M\to T^0\to C^1\to\cdots\to C^{n-1}\to C^n\to 0$$
in $\mathscr{A}$ with $T^0\in\mathscr{T}$ and all $C^i$ in $\mathscr{C}$;
that is, there exists an exact sequence
$$0\to M\to T\to K\to 0$$
in $\mathscr{A}$ with $T\in\mathscr{T}$ and $\mathscr{C}$-$\id K\leq n-1$.
\end{enumerate}
\end{enumerate}
\end{lem}

The main result in this subsection is as follows.

\begin{thm} \label{thm-3.2'}
Assume that $\mathscr{A}$ has enough injective objects and
$\mathscr{T}$ is a subcategory of $\mathscr{A}$ containing $\mathcal{I}(\mathscr{A})$.
Then the following statements are equivalent.
\begin{enumerate}
\item[$(1)$] $\mathscr{T}$ is coresolving.
\item[$(2)$] For any exact sequence
$$0 \to A_1 \to A_2 \to A_3 \to 0$$
in $\mathscr{A}$, we have
\begin{enumerate}
\item[$(i)$] $(a)$ $\mathscr{T}{\text -}\id A_2\leq
\max\{\mathscr{T}{\text -}\id A_1,\mathscr{T}{\text -}\id A_3\}$,
$(b)$ the equality holds if
$\mathscr{T}{\text -}\id A_1\neq\mathscr{T}{\text -}\id A_3$ $+1$.
\item[$(ii)$] $(a)$ $\mathscr{T}{\text -}\id A_3\leq
\max\{\mathscr{T}{\text -}\id A_1-1,\mathscr{T}{\text -}\id A_2\}$,
$(b)$ the equality holds if
$\mathscr{T}{\text -}\id A_1\neq\mathscr{T}{\text -}\id A_2$.
\item[$(iii)$] $(a)$ $\mathscr{T}{\text -}\id A_1\leq
\max\{\mathscr{T}{\text -}\id A_2,\mathscr{T}{\text -}\id A_3+1\}$,
$(b)$ the equality holds if
$\mathscr{T}{\text -}\id A_2\neq\mathscr{T}{\text -}\id A_3$.
\end{enumerate}
\end{enumerate}
\end{thm}

As an immediate consequence, we get the following result.

\begin{cor} \label{cor-3.3'}
Assume that $\mathscr{A}$ has enough injective objects and
$\mathscr{T}$ is a coresolving subcategory of $\mathscr{A}$. Then
$\mathscr{T}$-$\id^{<\infty}$ satisfies the two-out-of-three property; that is,
in a short exact sequence in $\mathscr{A}$,
if any two terms are in $\mathscr{T}$-$\id^{<\infty}$, then so is the third term.
\end{cor}

The following result shows that if the coresolving subcategory $\mathscr{T}$
of $\mathscr{A}$ admits an $\mathscr{E}$-proper
generator $\mathscr{C}$, then any object in $\mathscr{A}$ with finite
$\mathscr{T}$-injective dimension is isomorphic to a kernel (respectively, a cokernel)
of a morphism from an object in $\mathscr{T}$ to an object in $\mathscr{A}$
with finite $\mathscr{C}$-injective dimension.

\begin{cor}\label{cor-3.4'}
Let $\mathscr{E}$ be a subcategory of $\mathscr{A}$. If $\mathscr{T}$ is an $\mathscr{E}$-preresolving
subcategory of $\mathscr{A}$ admitting an $\mathscr{E}$-proper generator $\mathscr{C}$,
then for any $M\in\mathscr{A}$ with $\mathscr{T}{\text -}\id M=n<\infty$, we have
\begin{enumerate}
\item[$(1)$] There exists an exact sequence
$$0\to T^{'}\to K^{'}\to T\to K\to 0$$
in $\mathscr{A}$ with $\mathscr{C}{\text -}\id K^{'}\leq n$, $\mathscr{C}{\text -}\id K\leq n-1$
and $T^{'},T\in\mathscr{T}$, such that $M\cong\Im(K^{'}\to T)$.
\item[$(2)$] If $\mathscr{A}$ has enough injective objects and
$\mathscr{T}$ is coresolving in $\mathscr{A}$, then the two ``$\leq$" in (1) are ``$=$".
\end{enumerate}
\end{cor}

Furthermore, we get the following result.

\begin{cor} \label{cor-3.5'}
Assume that $\mathscr{A}$ has enough injective objects and
$\mathscr{T}$ is a coresolving subcategory of $\mathscr{A}$ admitting an $\mathscr{E}$-proper
generator $\mathscr{C}$. Then
\begin{enumerate}
\item[$(1)$] $\mathscr{T}{\text -}\FID\leq\mathscr{C}{\text -}\FID$.
\item[$(2)$] If $\mathscr{T}\subseteq{\mathscr{C}^{\bot}}$, then
$\mathscr{T}{\text -}\id M=\mathscr{C}{\text -}\id M$ for any $M\in\mathscr{A}$ with
$\mathscr{C}{\text -}\id M<\infty$.
\item[$(3)$] If $\mathscr{T}\subseteq{\mathscr{C}^{\bot}}$,
then $\mathscr{T}{\text -}\FID=\mathscr{C}{\text -}\FID$.
\end{enumerate}
\end{cor}

\begin{prop} \label{prop-3.6'}
Let $\mathscr{E}$ and $\mathscr{C}$ be subcategories of $\mathscr{A}$.
If $\mathscr{E}^{\bot}\cap\widetilde{\res_{\mathscr{E}}\mathscr{C}}$
is closed under ($\mathscr{E}$-proper) extensions, then it is closed under cokernels
of monomorphisms. In particular, if $\res{\mathscr{C}}:=\widetilde{\res_{\mathcal{P}(\mathscr{A})}{\mathscr{C}}}$
is closed under extensions, then it is closed under cokernels of monomorphisms.
\end{prop}

\begin{prop} \label{prop-3.7'}
Let $\mathscr{E}$ be a subcategory of $\mathscr{A}$. If $\mathscr{T}$ is an
$\mathscr{E}$-preresolving subcategory of $\mathscr{A}$ admitting an $\mathscr{E}$-proper
generator $\mathscr{C}$, then
$\widetilde{\res_{\mathscr{E}}\mathscr{C}}=\widetilde{\res_{\mathscr{E}}\mathscr{T}}$.
\end{prop}

\section {\bf Applications to module categories}

In this section, all rings are associative rings with unit and all modules are unital.
For a ring $R$, we use $\Mod R$ to denote the category of left $R$-modules and use $\mod R$ to
denote the category of finitely generated left $R$-modules.

\begin{df} \label{def-4.1} {\rm (\cite{ATY,HW}). Let $R$ and $S$ be arbitrary rings.
An ($R$-$S$)-bimodule $_RC_S$ is called
\textit{semidualizing} if the following conditions are satisfied.
\begin{enumerate}
\item[(a1)] $_RC$ admits a degreewise finite $R$-projective resolution.
\item[(a2)] $C_S$ admits a degreewise finite $S^{op}$-projective resolution.
\item[(b1)] The homothety map $_RR_R\stackrel{_R\gamma}{\rightarrow} \Hom_{S^{op}}(C,C)$ is an isomorphism.
\item[(b2)] The homothety map $_SS_S\stackrel{\gamma_S}{\rightarrow} \Hom_{R}(C,C)$ is an isomorphism.
\item[(c1)] $\Ext_{R}^{\geq 1}(C,C)=0$.
\item[(c2)] $\Ext_{S^{op}}^{\geq 1}(C,C)=0$.
\end{enumerate}}
\end{df}

Wakamatsu \cite{W1} introduced and studied the so-called {\it generalized tilting modules},
which are usually called {\it Wakamatsu tilting modules}, see \cite{BR, MR}. Note that
a bimodule $_RC_S$ is semidualizing if and only if it is Wakamatsu tilting (\cite[Corollary 3.2]{W3}).
Typical examples of semidualizing bimodules include the free module
of rank one and the dualizing module over a Cohen-Macaulay local ring.
More examples of semidualizing bimodules are referred to \cite{HW,TH1,W2}.

From now on, $R$ and $S$ are arbitrary rings and we fix a semidualizing bimodule $_RC_S$.
We write $(-)_*:=\Hom(C,-)$, and write
$$\mathcal{P}_C(R):=\{C\otimes_SP\mid P\ {\rm \ is\ projective\ in}\ \Mod S\},$$
$$\mathcal{F}_C(R):=\{C\otimes_SF\mid F\ {\rm \ is\ flat\ in}\ \Mod S\},$$
$$\mathcal{I}_C(R^{op}):=\{I_*\mid I\ {\rm \ is\ injective\ in}\ \Mod S^{op}\}.$$
The modules in $\mathcal{P}_C(R)$, $\mathcal{F}_C(R)$ and $\mathcal{I}_C(R^{op})$ are called {\it $C$-projective},
{\it $C$-flat} and {\it $C$-injective} respectively.
When ${_RC_S}={_RR_R}$, $C$-projective, $C$-flat and $C$-injective modules are exactly
projective, flat and injective modules respectively.

Let $\mathscr{B}$ be a subcategory of $\Mod R^{op}$. Recall that a sequence in $\Mod R$
is called {\it $(\mathscr{B}\otimes_R-)$-exact} if it is exact after applying the functor
$B\otimes_R-$ for any $B\in\mathscr{B}$. We write
$$\mathscr{B}^{\top}:=\{M\in\Mod R\mid \Tor_{\geq 1}^R(B,M)=0 {\rm\ for\ any}\ B\in\mathscr{B}\}.$$
The following notions were introduced by Holm and J$\phi$gensen \cite{HJ} for
commutative rings. The following are their non-commutative versions.

\begin{df} \label{def-4.2}
{\rm \begin{enumerate}
\item[]
\item[(1)] A module $M\in\Mod R$ is called {\it $C$-Gorenstein projective} if
$M\in{^{\bot}\mathcal{P}_C(R)}$ and there exists a $\Hom_R(-,\mathcal{P}_C(R))$-exact exact sequence
$$0\rightarrow M\rightarrow G^0\rightarrow G^1\rightarrow \cdots\rightarrow G^i\rightarrow \cdots$$
in $\Mod R$ with all $G^i$ in $\mathcal{P}_C(R)$.
\item[(2)] A module $M\in\Mod R$ is called {\it $C$-Gorenstein flat} if
$M\in\mathcal{I}_C(R^{op})^{\top}$ and there exists an $(\mathcal{I}_C(R^{op})\otimes_R-)$-exact exact sequence
$$0\rightarrow M\rightarrow Q^0\rightarrow Q^1\rightarrow \cdots\rightarrow Q^i\rightarrow \cdots$$
in $\Mod R$ with all $Q^i$ in $\mathcal{F}_C(R)$.
\item[(3)] A module $N\in\Mod R^{op}$ is called {\it $C$-Gorenstein injective} if
$N\in\mathcal{I}_C(R^{op})^{\bot}$ and there exists a $\Hom_{R^{op}}(\mathcal{I}_C(R^{op}),-)$-exact exact sequence
$$\cdots\rightarrow E_i\rightarrow \cdots\rightarrow E_1\rightarrow E_0\rightarrow N\rightarrow 0$$
in $\Mod R^{op}$ with all $E_i$ in $\mathcal{I}_C(R^{op})$.
\end{enumerate}}
\end{df}

We use $\mathcal{GP}_C(R)$ (resp. $\mathcal{GF}_C(R)$) to denote the subcategory of $\Mod R$
consisting of $C$-Gorenstein projective (resp. flat) modules, and use $\mathcal{GI}_C(R^{op})$
to denote the subcategory of $\Mod R^{op}$ consisting of $C$-Gorenstein injective modules.
When $_RC_S={_RR_R}$, $C$-Gorenstein projective, flat and injective modules are
exactly Gorenstein projective, flat and injective modules respectively (\cite{EJ,H});
in this case, we write
$$\mathcal{P}(R):=\mathcal{P}_C(R),\ \ \mathcal{I}(R^{op}):=\mathcal{I}_C(R^{op}),\ \ \mathcal{F}(R):=\mathcal{F}_C(R),$$
$$\mathcal{GP}(R):=\mathcal{GP}_C(R),\ \ \mathcal{GI}(R^{op}):=\mathcal{GI}_C(R^{op}),\ \ \mathcal{GF}(R):=\mathcal{GF}_C(R).$$

\begin{df} \label{def-4.3}
{\rm (\cite{HW})
\begin{enumerate}
\item[(1)] The {\it Auslander class} $\mathcal{A}_{C}(R^{op})$ with respect to $C$ consists of all modules $N$
in $\Mod R^{op}$ satisfying the following conditions.
\begin{enumerate}
\item[(a1)] $\Tor_{\geq 1}^R(N,C)=0$.
\item[(a2)] $\Ext_{S^{op}}^{\geq 1}(C,N\otimes_RC)=0$.
\item[(a3)] The canonical evaluation homomorphism
$$\mu_N:N\rightarrow (N\otimes_RC)_*$$
defined by $\mu_N(x)(c)=x\otimes c$ for any $x\in N$ and $c\in C$ is an isomorphism in $\Mod R^{op}$.
\end{enumerate}
\item[(2)] The {\it Bass class} $\mathcal{B}_C(R)$ with respect to $C$ consists of all modules $M$
in $\Mod R$ satisfying the following conditions.
\begin{enumerate}
\item[(b1)] $\Ext_R^{\geq 1}(C,M)=0$.
\item[(b2)] $\Tor_{\geq 1}^S(C,M_*)=0$.
\item[(b3)] The canonical evaluation homomorphism
$$\theta_M:C\otimes_SM_*\rightarrow M$$
defined by $\theta_M(c\otimes f)=f(c)$ for any $c\in C$ and $f\in M_*$ is an isomorphism in $\Mod R$.
\end{enumerate}
\end{enumerate}}
\end{df}

For a subcategory $\mathscr{X}$ of $\Mod R$ (or $\Mod R^{op}$), we write
$$\mathscr{X}^+:=\{X^+\mid X\in\mathscr{X}\},$$
where $(-)^+=\Hom_{\mathbb{Z}}(-,\mathbb{Q}/\mathbb{Z})$
with $\mathbb{Z}$ the additive group of integers and $\mathbb{Q}$ the additive group of rational numbers.
For simplicity, we write
$$\widetilde{\res\mathscr{C}}:=\widetilde{\res_{\mathscr{C}}\mathscr{C}}\ \ \ {\rm and}\ \ \
\widetilde{\cores\mathscr{C}}:=\widetilde{\cores_{\mathscr{C}}\mathscr{C}}.$$

In the following, we present a partial list of examples of how the results obtained in Section 3 can be applied.

\begin{rem} \label{rem-4.4}
{\rm \begin{enumerate}
\item[]
\item[(1)]
It is well known that $\mathcal{P}(R)$ and $\mathcal{F}(R)$ are resolving and
$\mathcal{I}(R)$ is coresolving in $\Mod R$.

\item[(2)]
Let $(\mathscr{U},\mathscr{V})$ be a hereditary cotorsion pair in $\Mod R$, and let
$\mathscr{C}:=\mathscr{U}\cap\mathscr{V}$ be its {\it kernel}. Then
\begin{enumerate}
\item[(a)]
$\mathscr{U}$ is resolving in $\Mod R$ admitting a $\mathscr{C}$-coproper cogenerator
$\mathscr{C}$ (\cite[Lemma 4.4]{SZH1}).
\item[(b)] Dually, $\mathscr{V}$ is coresolving in $\Mod R$
admitting a $\mathscr{C}$-proper generator $\mathscr{C}$.
\end{enumerate}

\item[(3)]
\begin{enumerate}
\item[(a)]
$$\mathcal{GP}_C(R)={^{\bot}\mathcal{P}_C(R)}\cap\widetilde{\cores\mathcal{P}_C(R)}$$
is resolving in $\Mod R$ admitting a $\mathcal{P}_C(R)$-coproper cogenerator $\mathcal{P}_C(R)$
(\cite[Example 3.2(2) and Proposition 3.3]{SZH1}). In particular,
$$\mathcal{GP}(R)={^{\bot}\mathcal{P}(R)}\cap\widetilde{\cores\mathcal{P}(R)}$$
is resolving in $\Mod R$ admitting a $\mathcal{P}(R)$-coproper cogenerator $\mathcal{P}(R)$.
\item[(b)] Dually,
$$\mathcal{GI}_C(R^{op})={\mathcal{I}_C(R^{op})^{\bot}}\cap\widetilde{\res\mathcal{I}_C(R^{op})}$$
is coresolving in $\Mod R^{op}$ admitting an $\mathcal{I}_C(R^{op})$-proper generator $\mathcal{I}_C(R^{op})$
(\cite[Example 3.2(2) and the dual of Proposition 3.3]{SZH1}). In particular,
$$\mathcal{GI}(R^{op})={\mathcal{I}(R^{op})^{\bot}}\cap\widetilde{\res\mathcal{I}(R^{op})}$$
is coresolving in $\Mod R^{op}$ admitting an $\mathcal{I}(R^{op})$-proper generator $\mathcal{I}(R^{op})$.
\item[(c)] Let $R$ be a left and right Noetherian ring, and let $p(R)$ be the subcategory of $\mod R$
consisting of projective modules.
Recall that a module $M\in\mod R$ is said to {\it have Gorenstein dimension zero} \cite{AB}
or be {\it totally reflexive} \cite{AM} if $M\in\mathcal{G}p(R)$, where
$$\mathcal{G}p(R)={^{\bot}{_RR}}\cap\widetilde{\cores p(R)},$$
which is resolving in $\mod R$ admitting a $p(R)$-coproper cogenerator $p(R)$.
\end{enumerate}

\item[(4)]
\begin{enumerate}
\item[(a)] Recall from \cite{DLM} that a module $M\in\Mod R$ is called {\it strongly Gorenstein flat}
if $M\in\mathcal{SGF}(R)$, where
$$\mathcal{SGF}(R)={^{\bot}\mathcal{F}(R)}\cap\widetilde{\cores_{\mathcal{F}(R)}\mathcal{P}(R)}.$$
It is trivial that ${^{\bot}\mathcal{F}(R)}$ is closed under extensions.
By the dual version of \cite[Lemma 8.2.1]{EJ} (cf. \cite[Horseshoe Lemma 1.7]{H}),
it is easy to see that $\mathcal{SGF}(R)$ is closed under extensions.
It follows from Proposition \ref{prop-3.6} that
$\mathcal{SGF}(R)$ is resolving in $\Mod R$ admitting an $\mathcal{F}(R)$-coproper cogenerator $\mathcal{P}(R)$,
which generalizes \cite[Proposition 2.10(1)(2)]{DLM}.

\item[(b)] Recall from \cite{M,St} that a module $M\in\Mod R$ is called
{\it FP-injective} (or {\it absolutely pure}) if $M\in\mathcal{FI}(R)$, where
$\mathcal{FI}(R):=\{M\in\Mod R\mid\Ext_R^1(X,M)=0$ for all finitely presented
left $R$-modules $X\}$. Recall from \cite{MD} that a module $M\in\Mod R$ is called
{\it Gorenstein FP-injective} if $M\in\mathcal{GFI}(R)$, where
$$\mathcal{GFI}(R)={\mathcal{FI}(R)^{\bot}}\cap\widetilde{\res_{\mathcal{FI}(R)}\mathcal{I}(R)}.$$
It is trivial that ${\mathcal{FI}(R)^{\bot}}$ is closed under extensions.
By \cite[Lemma 8.2.1]{EJ}, it is easy to see that $\mathcal{GFI}(R)$ is closed under extensions.
It follows from Proposition \ref{prop-3.6'} that
$\mathcal{GFI}(R)$ is coresolving in $\Mod R$ admitting an $\mathcal{FI}(R)$-proper generator $\mathcal{I}(R)$,
which generalizes \cite[Proposition 2.6(1)(2)]{MD}.
\end{enumerate}

\item[(5)]
\begin{enumerate}
\item[(a)] Recall from \cite{BGH} that a module $M\in\Mod R$ is called {\it level}
if $M\in\mathcal{L}(R)$, where $\mathcal{L}(R)=\{M\in\Mod R\mid\Tor^R_1(X,M)=0$ for all
right $R$-modules $X$ admitting a degreewise finite $R^{op}$-projective resolution$\}$;
also recall that a module $M\in\Mod R$ is called
{\it Gorenstein AC-projective} if $M\in\mathcal{GP}_{ac}(R)$, where
$$\mathcal{GP}_{ac}(R)={^{\bot}\mathcal{L}(R)}\cap\widetilde{\cores_{\mathcal{L}(R)}\mathcal{P}(R)}.$$
By \cite[Lemma 8.6]{BGH}, we have that
$\mathcal{GP}_{ac}(R)$ is resolving in $\Mod R$ admitting a $\mathcal{L}(R)$-coproper cogenerator $\mathcal{P}(R)$.

\item[(b)] Recall from \cite{BGH} that a module $M\in\Mod R$ is called {\it absolutely clean}
if $M\in\mathcal{AC}(R)$, where $\mathcal{AC}(R)=\{M\in\Mod R\mid\Ext_R^1(X,M)=0$ for all
left $R$-modules $X$ admitting a degreewise finite $R$-projective resolution$\}$;
also recall that a module $M\in\Mod R$ is called
{\it Gorenstein AC-injective} if $M\in\mathcal{GI}_{ac}(R)$, where
$$\mathcal{GI}_{ac}(R)={\mathcal{AC}(R)^{\bot}}\cap\widetilde{\res_{\mathcal{AC}(R)}\mathcal{I}(R)}.$$
By \cite[Lemma 5.6]{BGH}, we have that
$\mathcal{GI}_{ac}(R)$ is coresolving in $\Mod R$ admitting an $\mathcal{AC}(R)$-proper generator $\mathcal{I}(R)$.
\end{enumerate}

\item[(6)]
\begin{enumerate}
\item[(a)]
$$\mathcal{A}_C(R^{op})={^{\bot}\mathcal{I}_C(R^{op})}\cap\widetilde{\cores\mathcal{I}_C(R^{op})},$$
which is resolving in $\Mod R^{op}$ admitting an $\mathcal{I}_C(R^{op})$-coproper cogenerator $\mathcal{I}_C(R^{op})$
(\cite[Example 3.2(2) and Proposition 3.3]{SZH1}; also cf. \cite[Theorem 2]{HW}).
\item[(b)] Dually,
$$\mathcal{B}_C(R)={\mathcal{P}_C(R)^{\bot}}\cap\widetilde{\res\mathcal{P}_C(R)},$$
which is coresolving in $\Mod R$ admitting a $\mathcal{P}_C(R)$-proper generator $\mathcal{P}_C(R)$
(\cite[Example 3.2(2) and the dual of Proposition 3.3]{SZH1}; also cf. \cite[Theorem 6.1]{HW}).
\end{enumerate}

\item[(7)]
Let $\mathscr{B}$ be a subcategory of $\Mod R^{op}$. Recall from \cite{EIP} that a module $M\in\Mod R$ is called
{\it Gorenstein $\mathscr{B}$-flat} (respectively, {\it projectively coresolved Gorenstein $\mathscr{B}$-flat})
if $M\in\mathscr{B}^{\top}$ and there exists a $(\mathscr{B}\otimes_R-)$-exact
exact sequence
$$0\rightarrow M\rightarrow Q^0\rightarrow Q^1\rightarrow \cdots\rightarrow Q^i\rightarrow \cdots$$
in $\Mod R$ with all $Q^i$ in $\mathcal{F}(R)$ (respectively, $\mathcal{P}(R)$).
We use $\mathcal{GF}_{\mathscr{B}}(R)$ (respectively, $\mathcal{PGF}_{\mathscr{B}}(R)$) to denote
the subcategory of $\Mod R$ consisting of Gorenstein $\mathscr{B}$-flat modules
(respectively, projectively coresolved Gorenstein $\mathscr{B}$-flat modules).

Also recall from \cite{EIP} that $\mathscr{B}$ is {\it semi-definable} if $\mathscr{B}$ is closed under
direct products and its definable closure $<\mathscr{B}>$ (the smallest subcategory of $\Mod R^{op}$
containing $\mathscr{B}$ which is closed under direct products, direct limits and pure submodules) contains
a pure injective module $D$ such that any module in $<\mathscr{B}>$ is a pure submodule of some direct product
of copies of $D$.

Let $B\in\Mod R^{op}$, $M\in\Mod R$ and $n\geq 1$. By \cite[Lemma 2.16(a)(b)]{GT}, we have
$$(B\otimes_R-)^+\cong\Hom_{R}(-,B^+),\eqno{(4.1)}$$
$$[\Tor_n^R(B,M)]^+\cong\Ext_{R}^n(M,B^+).\eqno{(4.2)}$$
It yields that
$$\mathcal{GF}_{\mathscr{B}}(R)={^{\bot}(\mathscr{B}^+)}\cap\widetilde{\cores_{\mathscr{B}^+}\mathcal{F}(R)},$$
$$\mathcal{PGF}_{\mathscr{B}}(R)={^{\bot}(\mathscr{B}^+)}\cap\widetilde{\cores_{\mathscr{B}^+}\mathcal{P}(R)}.$$
By \cite[Theorem 2.8]{EIP}, we have that $\mathcal{PGF}_{\mathscr{B}}(R)$ is resolving in $\Mod R$ admitting an
$\mathcal{I}_C(R^{op})^+$-coproper cogenerator $\mathcal{P}(R)$.
When $\mathscr{B}=\mathcal{I}(R^{op})$,
projectively coresolved Gorenstein $\mathscr{B}$-flat modules are called {\it projectively coresolved Gorenstein flat}
(\cite{SS}); in this case, we write $\mathcal{PGF}(R):=\mathcal{PGF}_{\mathscr{B}}(R)$. We have
$\mathcal{P}(R)\subseteq\mathcal{PGF}(R)=\mathcal{SGF}(R)(R)\cap\mathcal{GF}(R)$ (\cite[Lemma 3]{I}).

On the other hand,
it follows from \cite[Theorem 2.12 and Corollary 2.14]{EIP} that if $\mathscr{B}$ is semi-definable,
then $\mathcal{GF}_{\mathscr{B}}(R)$ is resolving in $\Mod R$ admitting a $\mathscr{B}^+$-coproper cogenerator $\mathcal{F}(R)$.
In particular, $\mathcal{GF}(R)$ is resolving in $\Mod R$ admitting an $\mathcal{I}_C(R^{op})^+$-coproper cogenerator
$\mathcal{F}(R)$ (also cf. \cite[Theorem 4.11]{SS}).

\item[(8)]
By (4.1) and (4.2), we have that
$$\mathcal{GF}_C(R)={^{\bot}(\mathcal{I}_C(R^{op})^+)}\cap\widetilde{\cores_{\mathcal{I}_C(R^{op})^+}\mathcal{F}_C(R)},$$
which admits an $\mathcal{I}_C(R^{op})^+$-coproper cogenerator $\mathcal{F}_C(R)$. It is trivial that
$\mathcal{P}(R)\subseteq\mathcal{F}(R)\subseteq\mathcal{GF}_C(R)$.
By Proposition \ref{prop-3.6}, we have that if $\mathcal{GF}_C(R)$ is closed under
extensions, then it is resolving in $\Mod R$.
\end{enumerate}}
\end{rem}

\subsection{Finitistic dimensions}

In this subsection, $R$ is an arbitrary associative ring.

By Corollaries \ref{cor-3.5} and \ref{cor-3.5'} and Remark \ref{rem-4.4}(2), we immediately get the following result.

\begin{cor} \label{cor-4.5}
Let $(\mathscr{U},\mathscr{V})$ be a hereditary cotorsion pair in $\Mod R$ with
the kernel $\mathscr{C}$. Then
\begin{enumerate}
\item[$(1)$] For any $M\in\Mod R$ with $\mathscr{C}{\text -}\pd M<\infty$, we have
$$\mathscr{U}{\text -}\pd M=\mathscr{C}{\text -}\pd M.$$ Moreover, we have
$$\mathscr{U}{\text -}\FPD=\mathscr{C}{\text -}\FPD.$$
\item[$(2)$] For any $M\in\Mod R$ with $\mathscr{C}{\text -}\id M<\infty$, we have
$$\mathscr{V}{\text -}\id M=\mathscr{C}{\text -}\id M.$$ Moreover, we have
$$\mathscr{V}{\text -}\FID=\mathscr{C}{\text -}\FID.$$
\end{enumerate}
\end{cor}

Following the usual customary notation, we write
$$\pd_RM:={\mathcal{P}}(R){\text -}\pd M,\ \ \id_RM:={\mathcal{I}}(R){\text -}\id M,\ \ \fd_RM:={\mathcal{F}}(R){\text -}\pd M,$$
$$\Gpd_RM:={\mathcal{GP}}(R){\text -}\pd M,\ \ \Gid_RM:={\mathcal{GI}}(R){\text -}\id M,\ \ \Gfd_RM:={\mathcal{GF}}(R){\text -}\pd M,$$
$$\GCpd_RM:={\mathcal{GP}}_C(R){\text -}\pd M,\ \ \GCid_RM:={\mathcal{GI}}_C(R){\text -}\id M,\ \ \GCfd_RM:={\mathcal{GF}}_C(R){\text -}\pd M.$$

By Corollary \ref{cor-3.5} and Remark \ref{rem-4.4}(3)--(7), we immediately get the following result,
in which the assertion (2) extends \cite[Proposition 2.27 and Theorem 2.28]{H}, and the assertion (3) generalizes
\cite[Lemma 4.6]{X}.

\begin{cor} \label{cor-4.6}
\begin{enumerate}
\item[]
\item[$(1)$] For any $M\in\Mod R$ with $\mathcal{P}_C(R){\text -}\pd M<\infty$, we have
$$\GCpd_RM=\mathcal{P}_C(R){\text -}\pd M.$$ Moreover, we have
$$\mathcal{GP}_C(R){\text -}\FPD=\mathcal{P}_C(R){\text -}\FPD.$$
\item[$(2)$] For any $M\in\Mod R$ with $\pd_RM<\infty$, we have
$$\Gpd_RM=\mathcal{GP}_{ac}(R){\text -}\pd M=\mathcal{SGF}(R){\text -}\pd M=\mathcal{PGF}(R){\text -}\pd M=\pd_RM.$$
Moreover, we have
$$\mathcal{GP}(R){\text -}\FPD=\mathcal{GP}_{ac}(R){\text -}\FPD=\mathcal{SGF}(R){\text -}\FPD=
\mathcal{PGF}(R){\text -}\FPD=\mathcal{P}(R){\text -}\FPD.$$
\item[$(3)$] Let $R$ be a left and right Noetherian ring. Then for any $M\in\mod R$ with $\pd_RM<\infty$, we have
$$\mathcal{G}p(R){\text -}\pd_RM=\pd_RM.$$
Moreover, we have
$$\mathcal{G}p(R){\text -}\FPD=p(R){\text -}\FPD.$$
\item[$(4)$] For any $N\in\Mod R^{op}$ with $\mathcal{I}_C(R^{op}){\text -}\pd N<\infty$, we have
$$\mathcal{A}_C(R^{op}){\text -}\pd N=\mathcal{I}_C(R^{op}){\text -}\pd N.$$
Moreover, we have
$$\mathcal{A}_C(R^{op}){\text -}\FPD=\mathcal{I}_C(R^{op}){\text -}\FPD.$$
\end{enumerate}
\end{cor}

By Corollary \ref{cor-3.5'} and Remark \ref{rem-4.4}(3)--(6), we immediately get the following result,
in which the assertion (2) extends \cite[Theorem 2.29]{H}.

\begin{cor} \label{cor-4.7}
\begin{enumerate}
\item[]
\item[$(1)$] For any $M\in\Mod R$ with $\mathcal{I}_C(R){\text -}\id M<\infty$, we have
$$\GCid_RM=\mathcal{I}_C(R){\text -}\id M.$$
Moreover, we have
$$\mathcal{GI}_C(R){\text -}\FID=\mathcal{I}_C(R){\text -}\FID.$$
\item[$(2)$] For any $M\in\Mod R$ with $\id_RM<\infty$, we have
$$\Gid_RM=\mathcal{GI}_{ac}(R){\text -}\id M=\mathcal{GFI}(R){\text -}\id M=\id_RM.$$
Moreover, we have
$$\mathcal{GI}(R){\text -}\FID=\mathcal{GI}_{ac}(R){\text -}\FID=\mathcal{GFI}(R){\text -}\FID=\mathcal{I}(R){\text -}\FID.$$
\item[$(3)$] For any $M\in\Mod R$ with $\mathcal{P}_C(R){\text -}\id M<\infty$, we have
$$\mathcal{B}_C(R){\text -}\id M=\mathcal{P}_C(R){\text -}\id M.$$
Moreover, we have
$$\mathcal{B}_C(R){\text -}\FID=\mathcal{P}_C(R){\text -}\FID.$$
\end{enumerate}
\end{cor}

\subsection{Equivalent characterizations of Gorenstein rings}

In this subsection, $R$ is a left and right Noetherian ring and $n\geq 0$.
Recall that $R$ is called {\it $n$-Gorenstein} if $\id_RR=\id_{R^{op}}R\leq n$.

The following lemma plays a crucial role in the sequel.

\begin{lem} \label{lem-4.8}
Let $\mathscr{T}$ be an $\mathscr{E}$-precoresolving subcategory of $\Mod R$
admitting an $\mathscr{E}$-coproper cogenerator $\mathscr{C}$, where
$\mathscr{E}$ is a subcategory of $\Mod R$ and $\mathscr{C}\subseteq\mathcal{F}(R)$.
If $\mathscr{T}$-$\pd M\leq n$ for any $M\in\mod R$, then $\id_{R^{op}}R\leq n$.
\end{lem}

\begin{proof}
Let $M\in\mod R$. If $\mathscr{T}$-$\pd M\leq n$, then by assumption
and Corollary \ref{cor-3.4}(1), there exists an exact sequence
$$0\to M\to K^{'}\to T^{'}\to 0$$
in $\Mod R$ with $\mathscr{C}$-$\pd K^{'}\leq n$ and $T^{'}\in\mathscr{T}$.
Since $\mathscr{C}\subseteq\mathcal{F}(R)$, we have $\fd_RK^{'}\leq n$.
Thus $\id_{R^{op}}R\leq n$ by \cite[Lemma 3.8]{HuH}.
\end{proof}

Recall from Remark \ref{rem-4.4}(3)(4) that
$${^{\bot}\mathcal{P}(R)}\cap\widetilde{\cores\mathcal{P}(R)}=\mathcal{GP}(R)
\supseteq\mathcal{SGF}(R)={^{\bot}\mathcal{F}(R)}\cap\widetilde{\cores_{\mathcal{F}(R)}\mathcal{P}(R)}.$$
In terms of the projective dimensions relative to all six subcategories of $\Mod R$ that appear in
this relation, we give some equivalent characterizations of $n$-Gorenstein rings as follows.

\begin{thm} \label{thm-4.9}
The following statements are equivalent.
\begin{enumerate}
\item[$(1)$] $R$ is $n$-Gorenstein.
\item[$(2)$] $\Gpd_RM\leq n$ for any $M\in\Mod R$.
\item[$(2)^{op}$] $\Gpd_{R^{op}}N\leq n$ for any $N\in\Mod R^{op}$.
\item[$(3)$] $^{\bot}\mathcal{P}(R){\text -}\pd M\leq n$ and $^{\bot}\mathcal{P}(R^{op}){\text -}\pd N\leq n$
for any $M\in\Mod R$ and $N\in\Mod R^{op}$.
\item[$(4)$] $\widetilde{\cores\mathcal{P}(R)}{\text -}\pd M\leq n$ and
$\widetilde{\cores\mathcal{P}(R^{op})}{\text -}\pd N\leq n$ for any $M\in\Mod R$ and $N\in\Mod R^{op}$.
\item[$(5)$] $\mathcal{SGF}(R){\text -}\pd M\leq n$ for any $M\in\Mod R$.
\item[$(5)^{op}$] $\mathcal{SGF}(R){\text -}\pd N\leq n$ for any $N\in\Mod R^{op}$.
\item[$(6)$] $^{\bot}\mathcal{F}(R){\text -}\pd M\leq n$ and $^{\bot}\mathcal{F}(R^{op}){\text -}\pd N\leq n$
for any $M\in\Mod R$ and $N\in\Mod R^{op}$.
\item[$(7)$] $\widetilde{\cores_{\mathcal{F}(R)}\mathcal{P}(R)}{\text -}\pd M\leq n$ and
$\widetilde{\cores_{\mathcal{F}(R^{op})}\mathcal{P}(R^{op})}{\text -}\pd N\leq n$ for any $M\in\Mod R$
and $N\in\Mod R^{op}$.
\end{enumerate}
\end{thm}

\begin{proof}
The implications $(2)+(2)^{op}\Rightarrow (3)+(4)$, $(5)+(5)^{op}\Rightarrow (6)+(7)$,
$(5)\Rightarrow (2)$, $(5)^{op}\Rightarrow (2)^{op}$, $(6)\Rightarrow (3)$
and $(7)\Rightarrow (4)$ are trivial. By \cite[Theorem 11.5.1]{EJ}, we have $(1)\Rightarrow (2)+(2)^{op}$.

If $R$ is $n$-Gorenstein, then $\mathcal{GP}(R)=\mathcal{SGF}(R)$ and
$\mathcal{GP}(R^{op})=\mathcal{SGF}(R^{op})$ by \cite[Corollary 2.8]{DLM},
and thus $(1)\Rightarrow (5)+(5)^{op}$ holds true.

$(3)\Rightarrow (1)$ By (3) and dimension shifting, it is easy to see that
$$\Ext_R^{\geq n+1}(M,R)=0=\Ext_{R^{op}}^{\geq n+1}(N,R)$$
for any $M\in\Mod R$ and $N\in\Mod R^{op}$.
It implies $\id_{R}R\leq n$ and $\id_{R^{op}}R\leq n$.

$(2)\Rightarrow (1)$ By (2) and dimension shifting, it is easy to get $\Ext_R^{\geq n+1}(M,R)=0$
for any $M\in\Mod R$, and so $\id_RR\leq n$.
By \cite[Theorem 2.5]{H}, we have that $\mathcal{GP}(R)$ is resolving
in $\Mod R$ admitting a $\mathcal{P}(R)$-coproper cogenerator $\mathcal{P}(R)(\subseteq\mathcal{F}(R))$.
Thus $\id_{R^{op}}R\leq n$ by (2) and Lemma \ref{lem-4.8}.

Symmetrically, we get $(2)^{op}\Rightarrow (1)$.

$(4)\Rightarrow (1)$ By the dual version of \cite[Lemma 8.2.1]{EJ} (cf. \cite[Horseshoe Lemma 1.7]{H}),
we have that $\widetilde{\cores\mathcal{P}(R)}$
is closed under $\mathcal{P}(R)$-coproper extensions. Thus $\widetilde{\cores\mathcal{P}(R)}$ is a
$\mathcal{P}(R)$-precoresolving subcategory of $\Mod R$
admitting a $\mathcal{P}(R)$-coproper cogenerator $\mathcal{P}(R)(\subseteq\mathcal{F}(R))$.
Thus $\id_{R^{op}}R\leq n$ by (4) Lemma \ref{lem-4.8}. Symmetrically, we have $\id_{R}R\leq n$.
\end{proof}

The following result is a dual version of Lemma \ref{lem-4.8}.

\begin{lem} \label{lem-4.10}
Let $\mathscr{T}$ be an $\mathscr{E}$-preresolving subcategory of $\Mod R$
admitting an $\mathscr{E}$-proper generator $\mathscr{C}$, where
$\mathscr{E}$ is a subcategory of $\Mod R$ and $\mathscr{C}\subseteq\mathcal{I}(R)$.
If $\mathscr{T}$-$\id M\leq n$ for any $M\in\Mod R$, then $\id_{R}R\leq n$.
\end{lem}

\begin{proof}
Let $N\in\mod R^{op}$. Then $N^+\in\Mod R$ and $\mathscr{T}$-$\id N^+\leq n$ by assumption.
It follows from Corollary \ref{cor-3.4'}(1) that there exists an exact sequence
$$0\to T^{'}\to K^{'}\buildrel {f} \over\longrightarrow N^+\to 0$$
in $\Mod R$ with $T^{'}\in\mathscr{T}$ and $\mathscr{C}$-$\id K^{'}\leq n$.
Since $\mathscr{C}\subseteq\mathcal{I}(R)$, we have $\id_{R}K^{'}\leq n$.
It follows from \cite[Theorem 2.2]{F} that $\fd_{R^{op}}{K^{'}}^+\leq n$.

On the other hand, by \cite[Proposition 5.3.9]{EJ}, there exists a monomorphism
$\lambda:N\rightarrowtail N^{++}$ in $\Mod R^{op}$, and hence
$\lambda f^+:N\rightarrowtail {K^{'}}^+$ is also a monomorphism in $\Mod R^{op}$.
Thus $\id_{R}R\leq n$ by \cite[Lemma 3.8]{HuH}.
\end{proof}

Recall from Remark \ref{rem-4.4}(3) that
$$\mathcal{GI}(R)={\mathcal{I}(R)^{\bot}}\cap\widetilde{\res\mathcal{I}(R)}.$$
In terms of the injective dimensions relative to all three subcategories of $\Mod R$ that appear in
this equality, we give some equivalent characterizations of $n$-Gorenstein rings as follows.

\begin{thm} \label{thm-4.11}
The following statements are equivalent.
\begin{enumerate}
\item[$(1)$] $R$ is $n$-Gorenstein.
\item[$(2)$] $\Gid_RM\leq n$ for any $M\in\Mod R$.
\item[$(2)^{op}$] $\Gid_{R^{op}}N\leq n$ for any $N\in\Mod R^{op}$.
\item[$(3)$] $\mathcal{I}(R)^{\bot}{\text -}\id M\leq n$ and $\mathcal{I}(R^{op})^{\bot}{\text -}\id N\leq n$
for any $M\in\Mod R$ and $N\in\Mod R^{op}$.
\item[$(4)$] $\widetilde{\res\mathcal{I}(R)}{\text -}\id M\leq n$ and
$\widetilde{\res\mathcal{I}(R^{op})}{\text -}\id N\leq n$ for any $M\in\Mod R$ and $N\in\Mod R^{op}$.
\end{enumerate}
\end{thm}

\begin{proof}
The implications $(2)+(2)^{op}\Rightarrow (3)+(4)$ are trivial.
By \cite[Theorem 11.2.1]{EJ}, we have $(1)\Rightarrow (2)+(2)^{op}$.

$(3)\Rightarrow (1)$ By \cite[Theorem 2.1]{F}, we have $(R_R)^+\in\mathcal{I}(R)$ and $(_RR)^+\in\mathcal{I}(R^{op})$.
Then by (3) and dimension shifting, it is easy to see that
$$\Ext_R^{\geq n+1}((R_R)^+,M)=0=\Ext_{R^{op}}^{\geq n+1}((_RR)^+,N)$$
for any $M\in\Mod R$ and $N\in\Mod R^{op}$.
It implies $\fd_{R}(R_R)^+\leq\pd_{R}(R_R)^+\leq n$ and $\fd_{R^{op}}(_RR)^+\leq\pd_{R^{op}}(_RR)^+\leq n$.
It follows from \cite[Theorem 2.2]{F} that $\id_{R^{op}}R\leq n$ and $\id_{R}R\leq n$.

$(2)\Rightarrow (1)$ Similar to the proof of $(3)\Rightarrow (1)$, we have $\id_{R^{op}}R\leq n$.
By \cite[Theorem 2.6]{H}, we have that $\mathcal{GI}(R)$ is coresolving
in $\Mod R$ admitting an $\mathcal{I}(R)$-proper generator $\mathcal{I}(R)$.
Thus $\id_{R}R\leq n$ by (2) and Lemma \ref{lem-4.10}.

Symmetrically, we get $(2)^{op}\Rightarrow (1)$.

$(4)\Rightarrow (1)$ By \cite[Lemma 8.2.1]{EJ}, we have that $\widetilde{\res\mathcal{I}(R)}$
is closed under $\mathcal{I}(R)$-proper extensions. Thus $\widetilde{\res\mathcal{I}(R)}$ is an
$\mathcal{I}(R)$-preresolving subcategory of $\Mod R^{op}$
admitting an $\mathcal{I}(R)$-proper generator $\mathcal{I}(R^{op})$.
Thus $\id_{R}R\leq n$ by (4) and Lemma \ref{lem-4.10}. Symmetrically, we have $\id_{R^{op}}R\leq n$.
\end{proof}

Recall from \cite{EJ} that a module $M\in\Mod R$ is called {\it cotorsion} if $\Ext^1_R(F,M)=0$
for any $F\in\mathcal{F}(R)$ (equivalently, $M\in\mathcal{F}(R)^{\bot}$). We write
$$\mathcal{FC}(R):=\{\text{flat and cotorsion modules in}\ \Mod R\}.$$

\begin{lem} \label{lem-4.12}
\begin{enumerate}
\item[]
\item[$(1)$]
$\mathcal{I}(R^{op})^+$ is an $\mathcal{I}(R^{op})^+$-coproper cogenerator
and $\mathcal{FC}(R)$ is an $\mathcal{FC}(R)$-coproper cogenerator for $\mathcal{F}(R)$.
\item[$(2)$] We have
\begin{align*}
& \ \ \ \ \  \ \ \widetilde{\cores\mathcal{I}(R^{op})^+}=\widetilde{\cores_{\mathcal{I}(R^{op})^+}\mathcal{FC}(R)}
=\widetilde{\cores_{\mathcal{I}(R^{op})^+}\mathcal{F}(R)}\\
&\ \ \ =\widetilde{\cores\mathcal{FC}(R)}=\widetilde{\cores_{\mathcal{FC}(R)}\mathcal{F}(R)}
\supseteq\widetilde{\cores\mathcal{F}(R)}.
\end{align*}
Moreover, all of these subcategories except $\widetilde{\cores\mathcal{F}(R)}$ are closed under
$\mathcal{I}(R^{op})^+$-coproper extensions.
\end{enumerate}
\end{lem}

\begin{proof}
(1) It essentially follows from \cite[Proposition 4.4]{SZH2} and its proof. However, we still
give the proof in details.

Let $Q\in\mathcal{F}(R)$. By \cite[Corollary 2.21(b)]{GT}, there exists the following pure exact sequence
$$0\to Q\to Q^{++}\to Q^{++}/Q\to 0\eqno{(4.3)}$$
in $\Mod R$. Since $Q^+\in\mathcal{I}(R^{op})$ and $Q^{++}\in\mathcal{I}(R^{op})^+\cap\mathcal{F}(R)$
by \cite[Theorems 2.1 and 2.2]{F}, we have $Q^{++}/Q\in\mathcal{F}(R)$ by \cite[Lemma 5.2(a)]{HW}, and
so (4.3) is a $\Hom_{R}(-,\mathcal{I}(R^{op})^+)$-exact exact sequence by \cite[Lemma 4.13]{SZH2}.
It follows that $\mathcal{I}(R^{op})^+$ is an $\mathcal{I}(R^{op})^+$-coproper cogenerator for $\mathcal{F}(R)$.

Since $Q^{++}$ is pure injective by \cite[Proposition 5.3.7]{EJ}, we have $Q^{++}\in\mathcal{FC}(R)$
by \cite[Proposition 4.4(1)]{SZH2}. Notice that (4.3) is a $\Hom_{R}(-,\mathcal{FC}(R))$-exact exact sequence,
so $\mathcal{FC}(R)$ is an $\mathcal{FC}(R)$-coproper cogenerator for $\mathcal{F}(R)$.

(2) Since $\mathcal{I}(R^{op})^+\subseteq\mathcal{FC}(R)\subseteq\mathcal{F}(R)$
by \cite[Proposition 5.3.7]{EJ} and \cite[Lemma 4.13]{SZH2}, we have
$$\widetilde{\cores\mathcal{I}(R^{op})^+}\subseteq\widetilde{\cores_{\mathcal{I}(R^{op})^+}\mathcal{FC}(R)}
\subseteq\widetilde{\cores_{\mathcal{I}(R^{op})^+}\mathcal{F}(R)}\supseteq
\widetilde{\cores_{\mathcal{FC}(R)}\mathcal{F}(R)}\supseteq\widetilde{\cores\mathcal{F}(R)}.$$
By (1) and Proposition \ref{prop-3.7}, we have
$$\widetilde{\cores\mathcal{I}(R^{op})^+}=\widetilde{\cores_{\mathcal{I}(R^{op})^+}\mathcal{F}(R)}\ \ {\rm and}\ \
\widetilde{\cores\mathcal{FC}(R)}=\widetilde{\cores_{\mathcal{FC}(R)}\mathcal{F}(R)}.$$
Suppose that $M\in\widetilde{\cores_{\mathcal{I}(R^{op})^+}\mathcal{F}(R)}$
and $$0\to M \to F^0\to F^1\to\cdots\to F^i\to\cdots\eqno{(4.4)}$$
is a $\Hom_R(-,\mathcal{I}(R^{op})^+)$-exact exact sequence in $\Mod R$ with all $F^i$ flat.
Let $D\in\mathcal{FC}(R)$. Then $D^{++}\in\mathcal{I}(R^{op})^+$ by \cite[Theorem 2.1]{F}.
Since $D$ is pure injective by \cite[Proposition 4.4(1)]{SZH2}, $D$ is isomorphic to a direct summand of
$D^{++}$ by \cite[Theorem 2.27]{GT}. Notice that (4.4) is $\Hom_R(-,D^{++})$-exact, so it is also $\Hom_R(-,D)$-exact.
Thus $M\in\widetilde{\cores_{\mathcal{FC}(R)}\mathcal{F}(R)}$ and
$\widetilde{\cores_{\mathcal{I}(R^{op})^+}\mathcal{F}(R)}\subseteq
\widetilde{\cores_{\mathcal{FC}(R)}\mathcal{F}(R)}$.

Since $\mathcal{I}(R^{op})^+$ is closed under $\mathcal{I}(R^{op})^+$-coproper extensions by \cite[Horseshoe Lemma 1.7]{H},
the latter assertion follows.
\end{proof}

Recall from Remark \ref{rem-4.4}(7)(8) and \cite[Theorem 4.6]{SZH2} that
\begin{align*}
& {^{\bot}(\mathcal{I}(R^{op})^+)}\cap\widetilde{\cores_{\mathcal{I}(R^{op})^+}\mathcal{F}(R)}
={^{\bot}\mathcal{FC}(R)}\cap\widetilde{\cores_{\mathcal{FC}(R)}\mathcal{F}(R)}
={^{\bot}\mathcal{FC}(R)}\cap\widetilde{\cores\mathcal{FC}(R)}=\mathcal{GF}(R)\\
&\supseteq\mathcal{PGF}(R)={^{\bot}(\mathcal{I}(R^{op})^+)}\cap\widetilde{\cores_{\mathcal{I}(R^{op})^+}\mathcal{P}(R)}.
\end{align*}
In terms of the projective dimensions relative to $\widetilde{\cores\mathcal{F}(R)}$ and
all eight subcategories of $\Mod R$ that appear in the above relation,
we give some equivalent characterizations of $n$-Gorenstein rings as follows.

\begin{thm} \label{thm-4.13}
The following statements are equivalent.
\begin{enumerate}
\item[$(1)$] $R$ is $n$-Gorenstein.
\item[$(2)$] $\Gfd_RM\leq n$ for any $M\in\Mod R$.
\item[$(2)^{op}$] $\Gfd_{R^{op}}N\leq n$ for any $N\in\Mod R^{op}$.
\item[$(3)$] $^{\bot}(\mathcal{I}(R^{op})^+){\text -}\pd M\leq n$ and $^{\bot}(\mathcal{I}(R)^+){\text -}\pd N\leq n$
for any $M\in\Mod R$ and $N\in\Mod R^{op}$.
\item[$(4)$] $^{\bot}\mathcal{FC}(R){\text -}\pd M\leq n$ and $^{\bot}\mathcal{FC}(R^{op}){\text -}\pd N\leq n$
for any $M\in\Mod R$ and $N\in\Mod R^{op}$.
\item[$(5)$] $\widetilde{\cores_{\mathcal{I}(R^{op})^+}\mathcal{F}(R)}{\text -}\pd M\leq n$ and
$\widetilde{\cores_{\mathcal{I}(R)^+}\mathcal{F}(R^{op})}{\text -}\pd N\leq n$ for any $M\in\Mod R$ and $N\in\Mod R^{op}$.
\item[$(6)$] $\widetilde{\cores_{\mathcal{FC}(R)}\mathcal{F}(R)}{\text -}\pd M\leq n$ and
$\widetilde{\cores_{\mathcal{FC}(R^{op})}\mathcal{F}(R^{op})}{\text -}\pd N\leq n$ for any $M\in\Mod R$ and $N\in\Mod R^{op}$.
\item[$(7)$] $\widetilde{\cores\mathcal{FC}(R)}{\text -}\pd M\leq n$ and
$\widetilde{\cores\mathcal{FC}(R^{op})}{\text -}\pd N\leq n$ for any $M\in\Mod R$ and $N\in\Mod R^{op}$.
\item[$(8)$] $\widetilde{\cores\mathcal{F}(R)}{\text -}\pd M\leq n$ and
$\widetilde{\cores\mathcal{F}(R^{op})}{\text -}\pd N\leq n$ for any $M\in\Mod R$ and $N\in\Mod R^{op}$.
\item[$(9)$] $\mathcal{PGF}(R){\text -}\pd M\leq n$ for any $M\in\Mod R$.
\item[$(9)^{op}$] $\mathcal{PGF}(R^{op}){\text -}\pd N\leq n$ for any $N\in\Mod R^{op}$.
\item[$(10)$] $\widetilde{\cores_{\mathcal{I}(R^{op})^+}\mathcal{P}(R)}{\text -}\pd M\leq n$ and
$\widetilde{\cores_{\mathcal{I}(R)^+}\mathcal{P}(R^{op})}{\text -}\pd N\leq n$ for any $M\in\Mod R$ and $N\in\Mod R^{op}$.
\end{enumerate}
\end{thm}

\begin{proof}
The implications $(2)+(2)^{op}\Rightarrow (3)+(4)$, $(9)\Rightarrow (2)$, $(9)^{op}\Rightarrow (2)^{op}$ and
$(9)+(9)^{op}\Rightarrow (10)\Rightarrow (5)$ are trivial.
By Lemma \ref{lem-4.12}, we have $(5)\Leftrightarrow (6)\Leftrightarrow (7)\Leftarrow (8)$.

Since $\widetilde{\cores\mathcal{F}(R)}\supseteq\widetilde{\cores_{\mathcal{F}(R)}\mathcal{P}(R)}$
and $\widetilde{\cores\mathcal{F}(R^{op})}\supseteq\widetilde{\cores_{\mathcal{F}(R^{op})}\mathcal{P}(R^{op})}$,
we have $(1)\Rightarrow (8)$ by Theorem \ref{thm-4.9}.

By \cite[Theorem 2.2]{F} and \cite[Lemma 4.13]{SZH2}, we have $\mathcal{I}(R^{op})^+\subseteq\mathcal{FC}(R)$
and $\mathcal{I}(R)^+\subseteq\mathcal{FC}(R^{op})$. Thus $^{\bot}(\mathcal{I}(R^{op})^+)\supseteq{^{\bot}\mathcal{FC}(R)}$
and $^{\bot}(\mathcal{I}(R)^+)\supseteq{^{\bot}\mathcal{FC}(R^{op})}$, and the implication $(4)\Rightarrow (3)$ follows.


$(1)\Rightarrow (9)+(9)^{op}$ By (1) and \cite[Theorem 2]{I}, we have
$\mathcal{SGF}(R)=\mathcal{PGF}(R)$ and $\mathcal{SGF}(R^{op})=\mathcal{PGF}(R^{op})$.
Now the assertion follows from Theorem \ref{thm-4.9}.

$(3)\Rightarrow (1)$ By \cite[Theorem 2.1]{F}, we have $(_RR)^+\in\mathcal{I}(R^{op})$ and $(R_R)^+\in\mathcal{I}(R)$.
Then by (3) and dimension shifting, it is easy to see that
$$\Ext_R^{\geq n+1}(M,(_RR)^{++})=0=\Ext_{R^{op}}^{\geq n+1}(N,(R_R)^{++})$$
for any $M\in\Mod R$ and $N\in\Mod R^{op}$.
It implies $\id_{R}(_RR)^{++}\leq n$ and $\id_{R^{op}}(R_R)^{++}\leq n$. It follows from \cite[Theorems 2.1 and 2.2]{F}
that $\id_{R}R=\fd_{R^{op}}(_RR)^{+}\leq n$ and $\id_{R^{op}}R=\fd_{R}(R_R)^{+}\leq n$.

$(2)\Rightarrow (1)$ Similar to the proof of $(3)\Rightarrow (1)$, we have $\id_{R}R\leq n$.
By Remark \ref{rem-4.4}(7), we have that $\mathcal{GF}(R)$ is resolving and admits an
$\mathcal{I}_C(R^{op})^+$-coproper cogenerator $\mathcal{F}(R)$.
Thus $\id_{R^{op}}R\leq n$ by (2) and Lemma \ref{lem-4.8}.

Symmetrically, we get $(2)^{op}\Rightarrow (1)$.

$(5)\Rightarrow (1)$ 
It follows from Lemma \ref{lem-4.12}(2) that $\widetilde{\cores_{\mathcal{I}(R^{op})^+}\mathcal{F}(R)}$
is an $\mathcal{I}(R^{op})^+$-precoresolving subcategory of $\Mod R$
admitting an $\mathcal{I}(R^{op})^+$-coproper cogenerator $\mathcal{F}(R)$.
Thus $\id_{R^{op}}R\leq n$ by (5) and Lemma \ref{lem-4.8}. Symmetrically, we have $\id_{R}R\leq n$.
\end{proof}

\subsection{$C$-Gorenstein flat modules}

In this subsection, $R,S$ are arbitrary rings and $_RC_S$ is a semidualizing bimodule.

\begin{lem} \label{lem-4.14}
For any $M\in\Mod R$, we have $\fd_SM_*=\id_{S^{op}}M^+\otimes_RC$.
\end{lem}

\begin{proof}
By \cite[Lemma 2.16(c)]{GT}, we have
$$(M_*)^+\cong M^+\otimes_RC.$$
It follows from \cite[Theorem 2.1]{F} that
$$\fd_SM_*=\id_{S^{op}}(M_*)^+=\id_{S^{op}}M^+\otimes_RC.$$
\end{proof}

We also need the following observation.

\begin{lem}\label{lem-4.15}
Let $n\geq 0$. Then
\begin{enumerate}
\item[$(1)$] For any $M\in\Mod R$, we have
$$\mathcal{F}_C(R){\text -}\pd_RM\leq n\Leftrightarrow
M\in\mathcal{B}_C(R)\ {\rm and}\ \fd_SM_*\leq n.$$
\item[$(2)$] For any $N\in\Mod R^{op}$, we have
$$\mathcal{I}_C(R^{op}){\text -}\id_{R^{op}}N\leq n\Leftrightarrow
N\in\mathcal{A}_C(R^{op})\ {\rm and}\ \id_{S^{op}}N\otimes_RC\leq n.$$
\end{enumerate}
\end{lem}

\begin{proof}
By \cite[Corollary 6.1]{HW}, we have
$$\mathcal{F}_C(R){\text -}\pd^{<\infty}\subseteq\mathcal{B}_C(R)\ {\rm and}\
\mathcal{I}_C(R^{op}){\text -}\id^{<\infty}\subseteq\mathcal{A}_C(R^{op}).$$
Then the assertions follow from \cite[Lemma 2.6(1)(3)]{TH3}.
\end{proof}

For any $M\in\Mod R$, we have the following canonical evaluation homomorphism
$$\sigma_M:M\to M^{++}$$
defined by $\sigma_M(x)(\alpha)=\alpha(x)$ for any $x\in M$ and $\alpha\in M^+$.

\begin{lem} \label{lem-4.16}
\begin{enumerate}
\item[]
\item[$(1)$]
Let $I$ be an injective right $S$-module. Then $(I_*)^{++}\cong(I^{++})_*$.
Moreover, $(I_*)^+\in\mathcal{F}_C(R)$ if $S$ is a right coherent ring.
\item[$(2)$]
Let $f:M_1^+\to M_2^+$ be a homomorphism in $\Mod R^{op}$ with $M_1,M_2\in\Mod R$.
If $M_1$ is pure injective, then there exists a homomorphism $g:M_2\to M_1$ in $\Mod R$
such that $f=g^+$.
\end{enumerate}
\end{lem}

\begin{proof}
(1) Let $I$ be an injective right $S$-module. Then $(I_*)^{+}\cong C\otimes_SI^+$
by \cite[Lemma 2.16(c)]{GT}, and hence
$$(I_*)^{++}\cong(C\otimes_SI^+)^{+}\cong(I^{++})_*$$
by \cite[Lemma 2.16(a)]{GT}. If $S$ is a right coherent ring, then $I^+$ is a flat
left $S$-module by \cite[Theorem 1]{CS}, and hence $(I_*)^{+}\cong C\otimes_SI^+\in\mathcal{F}_C(R)$.

(2) Let $f:M_1^+\to M_2^+$ be a homomorphism in $\Mod R^{op}$ with $M_1,M_2\in\Mod R$.
If $M_1$ is pure injective, then $\sigma_{M_1}:M_1\to M_1^{++}$ is a split monomorphism in $\Mod R$
by \cite[Proposition 2.27]{GT}. So there exists a split epimorphism $\beta:M_1^{++}\to M_1$ in $\Mod R$
such that $\beta\sigma_{M_1}=1_{M_1}$, and hence $(\sigma_{M_1})^+\beta^+=1_{M_1^+}$.
On the other hand, we also have $(\sigma_{M_1})^+\sigma_{M_1^+}=1_{M_1^+}$ by
\cite[Proposition 20.14(1)]{AF}. It follows that $$\beta^+=\sigma_{M_1^+}.\eqno{(4.5)}$$
Since the following diagram
$$\xymatrix{
M_1^+ \ar[d]_{\sigma_{M_1^+}} \ar[r]^{f}& M_2^+ \ar[d]^{\sigma_{M_2^+}} \\
M_1^{+++} \ar[r]^{f^{++}} & M_2^{+++} }$$
is commutative, we have $\sigma_{M_2^+}f=f^{++}\sigma_{M_1^+}$. Then by
\cite[Proposition 20.14(1)]{AF} and (4.5), we have
$$f=1_{M_2^+}f=(\sigma_{M_2})^+\sigma_{M_2^+}f
=(\sigma_{M_2})^+f^{++}\sigma_{M_1^+}=(\sigma_{M_2})^+f^{++}\beta^+=(\beta f^+\sigma_{M_2})^+.$$
Set $g:=\beta f^+\sigma_{M_2}$. Then $f=g^+$.
\end{proof}

The assertions in the following result are the $C$-versions of \cite[Theorem 2.1]{F}
and \cite[Theorem 3.6]{H} respectively.

\begin{thm} \label{thm-4.17}
For any $M\in\Mod R$, we have
\begin{enumerate}
\item[$(1)$] $\mathcal{F}_C(R){\text -}\pd_RM=\mathcal{I}_C(R^{op}){\text -}\id_{R^{op}}M^+$.
\item[$(2)$] $\GCfd_RM\geq\GCid_{R^{op}}M^+$ with equality if $S$ is a right coherent ring.
\end{enumerate}
\end{thm}

\begin{proof}
(1) For any $n\geq 0$, we have
\begin{align*}
& \ \ \ \  \ \  \ \ \mathcal{F}_C(R){\text -}\pd_RM\leq n\\
&\ \ \ \Leftrightarrow M\in\mathcal{B}_C(R)\ {\rm and}\ \fd_SM_*\leq n
\ \ \text{(by Lemma \ref{lem-4.15}(1))}\\
&\ \ \ \Leftrightarrow M^+\in\mathcal{A}_C(R^{op})\ {\rm and}\ \id_{S^{op}}M^+\otimes_RC\leq n
\ \ \text{(by \cite[Proposition 3.2(b)]{Hu3} and Lemma \ref{lem-4.14})}\\
&\ \ \ \Leftrightarrow \mathcal{I}_C(R^{op}){\text -}\id_{R^{op}}M^+\leq n.
\ \ \text{(by Lemma \ref{lem-4.15}(2))}
\end{align*}

(2) Let $E\in\mathcal{I}_C(R^{op})$ and $n\geq 1$. By \cite[Lemma 2.16(a)(b)]{GT}, we have
$$(E\otimes_R-)^+\cong\Hom_{R^{op}}(E,(-)^+),\eqno{(4.6)}$$
$$[\Tor_n^R(E,-)]^+\cong\Ext_{R^{op}}^n(E,(-)^+).\eqno{(4.7)}$$
If $G\in\mathcal{GF}_C(R)$, then $G\in\mathcal{I}_C(R^{op})^{\top}$ and there exists an
$(\mathcal{I}_C(R^{op})\otimes_R-)$-exact exact sequence
$$0\rightarrow G\rightarrow Q^0\rightarrow Q^1\rightarrow \cdots\rightarrow Q^i\rightarrow \cdots$$
in $\Mod R$ with all $Q^i$ in $\mathcal{F}_C(R)$.
It follows from (1) and the above two isomorphisms that
$G^+\in{\mathcal{I}_C(R^{op})^{\bot}}\cap\widetilde{\res\mathcal{I}_C(R^{op})}$, and
thus $G^+\in\mathcal{GI}_C(R^{op})$ by Remark \ref{rem-4.4}(3)$(b)$.
Then it is easy to get $\GCfd_RM\geq\GCid_{R^{op}}M^+$ for any $M\in\Mod R$.

Now let $S$ be a right coherent ring and $G\in\Mod R$.

{\bf Claim.} If $G^+\in\mathcal{GI}_C(R^{op})$, then $G\in\mathcal{GF}_C(R)$.

By Remark \ref{rem-4.4}(3)$(b)$, we have
$G^+\in{\mathcal{I}_C(R^{op})^{\bot}}\cap\widetilde{\res\mathcal{I}_C(R^{op})}$.
It follows from (4.7) that $G\in{\mathcal{I}_C(R^{op})^{\top}}$.
In addition, there exists the following $\Hom_{R^{op}}(\mathcal{I}_C(R^{op}),-)$-exact exact sequence
$$\cdots\to(I_i)_*\to\cdots\to(I_1)_*\to(I_0)_*\to G^+\to 0\eqno{(4.8)}$$
in $\Mod R^{op}$ with all $I_i$ injective right $S$-modules. Set $K_i:=\Im((I_i)_*\to(I_{i-1})_*)$
for any $i\geq 1$. Since $I_0\oplus I_0^{'}\cong I_0^{++}$
for some injective right $S$-module $I_0^{'}$, from Lemma \ref{lem-4.16}(1) and the exact sequence (4.8)
we get the following
$\Hom_{R^{op}}(\mathcal{I}_C(R^{op}),-)$-exact short exact sequence
$$0\to K_1\oplus(I_0^{'})_*\to(I_0)_*\oplus(I_0^{'})_*(\cong((I_0)_*)^{++})\to G^+\to 0$$
in $\Mod R^{op}$. Similarly, since $(I_1\oplus I_0^{'})\oplus I_1^{'}\cong(I_1\oplus I_0^{'})^{++}$
for some injective right $S$-module $I_1^{'}$, from Lemma \ref{lem-4.16}(1) and the exact sequence (4.8)
we get the following $\Hom_{R^{op}}(\mathcal{I}_C(R^{op}),-)$-exact short exact sequence
$$0\to K_2\oplus(I_1^{'})_*\to(I_1)_*\oplus(I_0^{'})_*\oplus(I_1^{'})_*
(\cong((I_1\oplus I_0^{'})_*)^{++})\to K_1\oplus(I_0^{'})_*\to 0$$
in $\Mod R^{op}$. Continuing this process and splicing these obtained short exact sequences, we get
the following $\Hom_{R^{op}}(\mathcal{I}_C(R^{op}),-)$-exact exact sequence
$$\cdots\to((I_i\oplus I_{i-1}^{'})_*)^{++}\to\cdots\to((I_1\oplus I_0^{'})_*)^{++}\to((I_0)_*)^{++}\to G^+\to 0
\eqno{(4.9)}$$
in $\Mod R^{op}$ with all $I_i^{'}$ injective right $S$-modules. Since $(I_0)_*)^{+}$
and all $(I_i\oplus I_{i-1}^{'})_*)^{+}$
are pure injective by \cite[Proposition 5.3.7]{EJ}, according to Lemma \ref{lem-4.16}(2) we can rewrite (4.9) as follows:
$$\cdots\to((I_i\oplus I_{i-1}^{'})_*)^{++}\buildrel{(g_i)^+}\over\longrightarrow\cdots\to((I_1\oplus I_0^{'})_*)^{++}
\buildrel{(g_1)^+}\over\longrightarrow((I_0)_*)^{++}\buildrel{(g_0)^+}\over\longrightarrow G^+\to 0.$$
Then by (4.6), we get the following $(\mathcal{I}_C(R^{op})\otimes_R-)$-exact exact sequence
$$0\to G\buildrel{g_0}\over\longrightarrow((I_0)_*)^{+}\buildrel{g_1}\over\longrightarrow((I_1\oplus I_0^{'})_*)^{+}
\to\cdots\buildrel{g_i}\over\longrightarrow((I_i\oplus I_{i-1}^{'})_*)^{+}\to\cdots$$
in $\Mod R$. By Lemma \ref{lem-4.16}(1), we have that $((I_0)_*)^{+}$ and all $((I_i\oplus I_{i-1}^{'})_*)^{+}$
are in $\mathcal{F}_C(R)$. Consequently we conclude that $G\in\mathcal{GF}_C(R)$. The claim is proved.

Let $M\in\Mod R$ with $\GCid_{R^{op}}M^+=n<\infty$, and let
$$0\to K_n\to G_{n-1}\to\cdots\to G_1\to G_0\to M\to 0$$
be an exact sequence in $\Mod R$ with all $G_i$ in $\mathcal{GF}_C(R)$.
Then we get the following exact sequence
$$0\to M^+\to G_0^+\to G_1^+\cdots\to G_{n-1}^+\to K_n^+\to 0$$
in $\Mod R^{op}$. By the former assertion, all $G_i^+$ are in $\mathcal{GI}_C(R^{op})$.
It follows from Remark \ref{rem-4.4}(3)$(b)$ and Lemma \ref{lem-3.1'}(1) that
$K_n^+\in\mathcal{GI}_C(R^{op})$. Then $K_n\in\mathcal{GF}_C(R)$ by the above claim,
and thus $\GCfd_RM\leq n$.
\end{proof}

As a consequence, we get the following result, in which the assertion (1) generalizes
\cite[Lemma 5.2(a)]{HW}.

\begin{cor} \label{cor-4.18}
For any $n\geq 0$, we have
\begin{enumerate}
\item[$(1)$]
The class of left $R$-modules with $\mathcal{F}_C(R)$-projective dimension at most $n$
is closed under pure submodules and pure quotients;
in particular, the class $\mathcal{F}_C(R)$ is closed under pure submodules and pure quotients.
\item[$(2)$]
If $S$ is a right coherent ring, then the class of left $R$-modules with
$\mathcal{GF}_C(R)$-projective dimension at most $n$ is closed under pure submodules and pure quotients;
in particular, the class $\mathcal{GF}_C(R)$ is closed
under pure submodules and pure quotients.
\end{enumerate}
\end{cor}

\begin{proof}
(1) Let
$$0\to K\to G \to L\to 0$$
be a pure exact sequence in $\Mod R$ with $\mathcal{F}_C(R)$-$\pd_RG\leq n$.
Then by \cite[Proposition 5.3.8]{EJ}, the induced exact sequence
$$0\to L^+\to G^+\to K^+\to 0$$
splits and both $K^+$ and $L^+$ are direct summands of $G^+$.
By Theorem \ref{thm-4.17}(1), we have $\mathcal{I}_C(R^{op})$-$\id_{R^{op}}G^+\leq n$.
Since $\mathcal{I}_C(R^{op})$ is closed under direct summands by \cite[Proposition 5.1(c)]{HW},
the class of right $R$-modules with $\mathcal{I}_C(R^{op})$-injective dimension at most $n$
is closed under direct summands by \cite[Corollary 4.9]{Hu2}. It follows that
$\mathcal{I}_C(R^{op})$-$\id_{R^{op}}K^+\leq n$ and $\mathcal{I}_C(R^{op})$-$\id_{R^{op}}L^+\leq n$.
Thus $\mathcal{F}_C(R)$-$\pd_RK\leq n$ and $\mathcal{F}_C(R)$-$\pd_RL\leq n$
by Theorem \ref{thm-4.17}(1) again.

(2) It is trivial that $\mathcal{I}_C(R^{op})^{\bot}$ is closed under direct summands.
By \cite[Theorem 4.6(1)]{Hu1}, the class $\widetilde{\res\mathcal{I}_C(R^{op})}$ is
closed under direct summands. Notice that
$$\mathcal{GI}_C(R^{op})={\mathcal{I}_C(R^{op})^{\bot}}\cap\widetilde{\res\mathcal{I}_C(R^{op})}$$
by Remark \ref{rem-4.4}(3)(b), thus $\mathcal{GI}_C(R^{op})$ is also closed under direct summands.
We also know from Remark \ref{rem-4.4}(3)(b) that $\mathcal{GI}_C(R^{op})$ is coresolving in $\Mod R^{op}$.
Thus the class of right $R$-modules with $\mathcal{GI}_C(R^{op})$-injective dimension at most
$n$ is closed under direct summands by \cite[Corollary 4.9]{Hu2}.
Now applying Theorem \ref{thm-4.17}(2), we obtain the assertion by using an argument
similar to that in the proof of (1).
\end{proof}

In the following result, the assertion (1) is the $C$-version of \cite[Theorem 2.2]{B2}.
The assertion (3) means that the assumption ``$R$ is a right coherent ring"
in \cite[Theorem 3.24]{H} is superfluous; compare it with Corollaries \ref{cor-4.6}(2) and \ref{cor-4.7}(2).

\begin{thm} \label{thm-4.19}
\begin{enumerate}
\item[]
\item[$(1)$]
For any $M\in\Mod R$, we have
$$\GCfd_RM\leq\mathcal{F}_C(R){\text -}\pd_RM$$
with equality if $\mathcal{F}_C(R){\text -}\pd_RM<\infty$.
\item[$(2)$]
$\mathcal{F}_C(R){\text -}\FPD\leq\mathcal{GF}_C(R){\text -}\FPD$
with equality if $\mathcal{GF}_C(R)$ is closed under extensions.
\item[$(3)$] $\mathcal{F}(R){\text -}\FPD=\mathcal{GF}(R){\text -}\FPD$.
\end{enumerate}
\end{thm}

\begin{proof}
(1) Since $\mathcal{GF}_C(R)\subseteq\mathcal{F}_C(R)$, we have
$\GCfd_RM\leq\mathcal{F}_C(R){\text -}\pd_RM$ for any $M\in\Mod R$.
Now let $\mathcal{F}_C(R){\text -}\pd_RM<\infty$. Then
\begin{align*}
& \ \ \ \  \ \  \ \ \mathcal{I}_C(R^{op}){\text -}\id_{R^{op}}M^+<\infty
\ \ \text{(by Theorem \ref{thm-4.17}(1))}\\
&\ \ \ \Rightarrow \GCid_{R^{op}}M^+=
\mathcal{I}_C(R^{op}){\text -}\id_{R^{op}}M^+\ \ \text{(by Corollary \ref{cor-4.7}(1))}\\
&\ \ \ \Rightarrow \GCfd_RM\geq\mathcal{F}_C(R){\text -}\pd_RM
\ \ \text{(by Theorem \ref{thm-4.17})}\\
&\ \ \ \Rightarrow \GCfd_RM=\mathcal{F}_C(R){\text -}\pd_RM.
\end{align*}

(2) The assertion that $\mathcal{F}_C(R){\text -}\FPD\leq\mathcal{GF}_C(R){\text -}\FPD$
follows from (1).

It is trivial that $\mathcal{P}(R)\subseteq\mathcal{F}(R)\subseteq\mathcal{GF}_C(R)$.
By Remark \ref{rem-4.4}(8), we have that
$$\mathcal{GF}_C(R)={^{\bot}(\mathcal{I}_C(R^{op})^+)}\cap\widetilde{\cores_{\mathcal{I}_C(R^{op})^+}\mathcal{F}_C(R)}$$
and it admits an $\mathcal{I}_C(R^{op})^+$-coproper cogenerator $\mathcal{F}_C(R)$.
If $\mathcal{GF}_C(R)$ is closed under extensions, then $\mathcal{GF}_C(R)$ is resolving in $\Mod R$ by Proposition \ref{prop-3.6}.
Now let $M\in\Mod R$ with $\GCfd_RM=n<\infty$.
By Corollary \ref{cor-3.4}(2), there exists an exact sequence
$$0\to M\to K^{'}\to T^{'}\to 0$$
in $\Mod R$ with $\mathcal{F}_C(R){\text -}\pd K^{'}=n$.
It follows that $\mathcal{GF}_C(R){\text -}\FPD\leq\mathcal{F}_C(R){\text -}\FPD$.

(3) Since $\mathcal{GF}(R)$ is closed under extensions by \cite[Theorem 4.11]{SS},
the assertion follows from (2) by putting ${_RC_S}={_RR_R}$.
\end{proof}

\end{document}